\documentclass[11pt]{amsart}

\usepackage{amsmath}
\usepackage{amssymb}
\usepackage{amsfonts}
\usepackage{amsthm}
\usepackage{enumerate}
\usepackage[mathscr]{eucal}

\usepackage{graphicx}
\usepackage{verbatim}

\newcommand{\vbf}{\bf}
\newcommand{\ls}{\lesssim}

\newcommand{\E}{\mathcal E}

\newcommand{\ci}[1]{_{{}_{\!\scriptstyle{#1}}}}

\newcommand{\Be}{\begin{equation}}
\newcommand{\Ee}{\end{equation}}

\newcommand{\Bea}{\begin{eqnarray}}
\newcommand{\Eea}{\end{eqnarray}}
\newcommand{\Beas}{\begin{eqnarray*}}
\newcommand{\Eeas}{\end{eqnarray*}}
\newcommand{\Benu}{\begin{enumerate}}
\newcommand{\Eenu}{\end{enumerate}}
\newcommand{\Bi}{\begin{itemize}}
\newcommand{\Ei}{\end{itemize}}

\def\Jscr{{\mathcal J}}
\def\sh{{\text{sh}}}

\def\lg{{\text{lg}}}

\def\h{h}

\def\intslash{\rlap{\kern  .32em $\mspace {.5mu}\backslash$ }\int}
\def\qsl{{\rlap{\kern  .32em $\mspace {.5mu}\backslash$ }\int_{Q_x}}}
\def\Re{\operatorname{Re\,}}

\def\Q{\mathcal Q}

\def\emph#1{{\it #1 }}

\def\cf{{\it cf}}

\def\dist{{\text{\it dist}}}

\def\supp{{\text{\rm supp}}}
\def\rad{{\text{\it rad}}}

\def\inn#1#2{\langle#1,#2\rangle}
\def\biginn#1#2{\big\langle#1,#2\big\rangle}

\def\noi{\noindent}

\def\meas{{\text{\rm meas}}}

\def\lc{\lesssim}

\def\eps{\varepsilon}

\def\ka{\kappa}
            
\def\la{\lambda}

              \def\Om{\Omega}

\def\fQ{{\mathfrak {Q}}}

\def\fS{{\mathfrak {S}}}

\def\fW{{\mathfrak {W}}}

\def\fZ{{\mathfrak {Z}}}

\def\fm{{\mathfrak {m}}}

\def\fz{{\mathfrak {z}}}

\def\bbC{{\mathbb {C}}}

\def\bbN{{\mathbb {N}}}

\def\bbR{{\mathbb {R}}}

\def\bbZ{{\mathbb {Z}}}

\def\sH{{\mathscr {H}}}

\def\sK{{\mathscr {K}}}

\def\cB{{\mathcal {B}}}

\def\cE{{\mathcal {E}}}
\def\cF{{\mathcal {F}}}
\def\cG{{\mathcal {G}}}
\def\cH{{\mathcal {H}}}

\def\cL{{\mathcal {L}}}
\def\cM{{\mathcal {M}}}

\def\cQ{{\mathcal {Q}}}

\def\cT{{\mathcal {T}}}

\def\cW{{\mathcal {W}}}

\def\Q{{\hbox{\bf Q}}}

 at 10 true pt

\def\be#1{\begin{equation}\label{ #1}}
\def\endeq{\end{equation}}
\def\endal{\end{align}}
\def\bas{\begin{align*}}
\def\eas{\end{align*}}
\def\bi{\begin{itemize}}
\def\ei{\end{itemize}}

\def\eps{\varepsilon}

\def\emph#1{{\it #1}}
\def\textbf#1{{\bf #1}}

\theoremstyle{plain}
  \newtheorem{theorem}{Theorem}[section]

   \newtheorem{proposition}[theorem]{Proposition}
   
   \newtheorem{lemma}[theorem]{Lemma}
   \newtheorem{corollary}[theorem]{Corollary}

\theoremstyle{remark}
   \newtheorem{remark}[theorem]{Remark}

\theoremstyle{definition}
   
   \newtheorem{hypothesis}[theorem]{Hypothesis}

\begin{document}

\title[Square functions and maximal operators]{Square functions and maximal 
operators \\ associated with radial Fourier multipliers}

\dedicatory{Dedicated to Eli Stein}

\author{Sanghyuk Lee \ \ \ \ Keith M. Rogers \ \ \ \ Andreas Seeger}


\thanks{Supported in part by NRF grant 2012008373,
ERC grant 277778,
MINECO grants MTM2010-16518, SEV-2011-0087  and NSF grant 1200261}

\maketitle



We begin with an overview  on   square functions for 
spherical and Bochner--Riesz means which were introduced by Eli Stein, 
 and discuss  their implications for radial multipliers and associated maximal functions. We then prove new endpoint estimates for these square functions,  for the maximal Bochner--Riesz operator, 
and for  more general 
classes of 
radial Fourier multipliers.

\section*{Overview}
\subsubsection*{Square functions}
The classical Littlewood--Paley
functions on $\bbR^d$ are  defined by
\[g[f]= \Big(\int_0^\infty \Big|\frac{\partial}{\partial t}
P_t f\Big|^2t\,dt
\Big)^{1/2}
\]
where $(P_t)_{t>0}$   is an approximation of the identity defined
by the dilates of a \lq nice\rq \, kernel (for example $(P_t)$ may be
the Poisson or the heat semigroup).
Their  significance in harmonic analysis, and many important variants and generalizations
have been discussed in  Stein's monographs \cite{stein70}, \cite{stein-si},
\cite{steinharmonic}, in the survey \cite{stein-wainger}
by Stein and Wainger,  and in the  historical article \cite{stein85zyg}.

Here we focus on $L^p$-bounds for
 two square functions introduced by Stein, for which
$(P_t)$ is replaced by a  family of operators with rougher kernels
or multipliers.
The first  is  generated by the  {\it
generalized  spherical means}
$$
A^\beta_t f(x) =
\frac{1}{\Gamma(\beta)}
\int_{|y|\le 1}\,
(1-|y|^2)^{\beta-1} f(x-ty)\,dy
$$
defined {\it a priori} for $\Re \beta>0$. The definition can be extended
to $\Re \beta\le 0$  by analytic continuation;
for $\beta=0$ we recover the standard
 spherical means.
In    \cite{stein76}  Stein
used  (a variant of)
the square
function
$$ \mathcal G_\beta f=
\Big(\int_0^\infty \Big| \frac{\partial}{\partial t}
A^{\beta}_t \!f\Big|^2
t\,dt\Big)^{1/2}
$$
to prove $L^p$-estimates
for  the maximal function
$\sup_{t>0} |A^{\beta-1/2+\eps}_t  f|$,
in parti\-cular he established pointwise convergence for the standard spherical means when $p>\frac{d}{d-1}$ and $d\ge 3$; see also \cite{stein-wainger}.

The  second square function
$$G_\alpha f =\Big(\int_0^\infty \Big| \frac{\partial}{\partial t}
R^{\alpha}_t \!f\Big|^2 t\,dt\Big)^{1/2},
$$
generated by the
Bochner--Riesz means
$$ R^\alpha_t f(x)
= \frac{1}{(2\pi)^{d}}\int_{|\xi|\le t}\Big(1-\frac{|\xi|^2}{t^2}\Big)^\alpha\widehat f(\xi) \,e^{i\inn x\xi} \,d\xi\,,$$
 was introduced  in Stein's 1958 paper
\cite{stein58} and used to control
the maximal function
$\sup_{t>0} |R^{\alpha-1/2+\eps}_t  f|$ for $f\in L^2$ in order to
prove almost everywhere convergence for Bochner--Riesz means of both
Fourier integrals and series (see also Chapter VII in \cite{stw}).
Later, starting with the work of Carbery  \cite{carbery}, it was recognized that sharp $L^p$ bounds for $G_\alpha$ with $p>2$ imply 
 sharp  $L^p$-bounds for maximal functions
associated with  Bochner--Riesz means and then also 
maximal functions  associated with more general classes of
 radial Fourier multipliers (\cite{caesc}, \cite{dt}).

In \cite{stein-wainger},
Stein and Wainger posed the problem of investigating
the relationships between various square functions. Addressing this
problem,
Sunouchi \cite{su2} (in one dimension) and
Kaneko and Sunouchi \cite{ks} (in higher dimensions)
used Plancherel's theorem to establish among other things the
   uniform pointwise  equivalence
\Be\label{ptwequiv}
G_\alpha f(x) \approx
\cG_\beta f(x), \quad \beta=\alpha-\tfrac{d-2}{2}.
\Ee
In view of this remarkable result we  shall  consider
$G_\alpha$ only.

\subsubsection*{Implications for radial multipliers} \label{radialFmsection}
We   recall Stein's point of view for proving results for Fourier
multipliers from Littlewood--Paley theory.
Suppose the convolution operator $\cT$  is given by
$\widehat{\cT f}=h\widehat f$
where $h$ satisfies the assumptions of the H\"ormander multiplier theorem.
 That is,
if $\varphi$ is a radial nontrivial $C^\infty$ function with compact support away from the origin,
 and $L^2_\alpha(\bbR^d)$ is the usual Sobolev space, it is assumed that
$\sup_{t>0} \|\varphi h(t\,\cdot\,)\|_{L^2_\alpha}$ is finite for some $\alpha>d/2$.
Under this assumption $\cT$ is bounded on $L^p$ for $1<p<\infty$
(\cite{hoer}, \cite{stein-si}, \cite{triebel}).
In Chapter IV of  the monograph  \cite{stein-si},
 Stein approached this result
by establishing the pointwise inequality
\Be\label{steinptw}
g[\cT\!f](x)  \le C\sup_{t>0} \|\varphi h(t\,\cdot\,)\|_{L^2_\alpha} \,g_{\la(\alpha)}^*[f],
\Ee where $g$ is a standard  Littlewood--Paley function and
$g_{\la}^*$ is a  tangential  variant of $g$ which does not depend on the specific multiplier.  As
$\|g_{\la(\alpha)}^* [f]\|_p\lc \|f\|_p$ for $2\le p<\infty$ and $\alpha>d/2$,
this proves the theorem since (under certain nondegeneracy assumptions on the generating kernel) one also has $\|g[f]\|_p\approx\|f\|_p$ for $1<p<\infty$.

A similar point of view was later used  for {\it radial} Fourier
multipliers. Let $m$ be a bounded function
on $\bbR^+$, let  $\varphi_\circ\in C^\infty_0(1,2)$, and let
$T_m$ be defined by
\Be\label{defofTm}
\widehat {T_m f}(\xi)  =m(|\xi|) \widehat f(\xi)\,.
\Ee
The work of Carbery, Gasper and Trebels \cite{cgt} yields an analogue of 
\eqref{steinptw} for radial multipliers in which the $g_\la^*$-function is  replaced  with a robust 
version  of $G_\alpha$ which has the same $L^p$ boundedness properties as $G_\alpha$. A variant   of their 
argument,
given  by Carbery in 
\cite{caesc},  shows that one can work with $G_\alpha$ itself and so 
there is a pointwise estimate
\Be\label{cgtpointwise}
g[T_mf](x)  \le C\sup_{t>0}
\|\varphi_\circ m(t\cdot)\|_{L^2_\alpha(\bbR)}\, G_\alpha f(x)\,
\Ee
where again $g$ is a suitable standard Littlewood--Paley function.
$L^p$  mapping properties 
of $G_\alpha$ together with \eqref{cgtpointwise} have been used 
to prove essentially sharp 
 estimates for radial convolution operators, with multipliers in 
localized Sobolev spaces. However it was not evident whether 
\eqref{cgtpointwise} could also be used to capture endpoint results, for radial 
multipliers in the same family of spaces. We shall address this point  in 
\S\ref{endptsect} below.

Carbery \cite{caesc} 
also obtained a related pointwise inequality  for maximal functions, 
\Be\label{maximaloperatorineq}
\sup_{t>0} |T_{m(t\cdot)}f(x)| \le C \|m\circ \exp\|_{L^2_\alpha(\bbR)}\,\, G_\alpha f(x)\,,
\Ee
which for $p\ge 2$ yields  effective $L^p$ bounds for maximal 
operators 
generated by radial Fourier multipliers from such  bounds for $G_\alpha$;  
see also Dappa and Trebels \cite{dt} for similar  results. 
Only little is currently known 
about  maximal operators for radial Fourier multipliers in the range $p<2$;
 \cf. Tao's work
\cite{taowt}, \cite{taomax} for examples and for partial results in two dimensions.

\subsubsection*{ $L^p$-bounds for $G_\alpha$}\label{Lpboundssect}
We now discuss necessary conditions and sufficient conditions on $p\in (1,\infty)$ for the validity of the  inequality
\begin{equation}
\label{Galphaineq}
\|G_\alpha f\|_p\lc \|f\|_p\,;\end{equation}
here the notation $A\lc B$ is used for $A\le CB$ with an unspecified constant.
By \eqref{cgtpointwise} it is necessary for
\eqref{Galphaineq}
that $\alpha>1/2$
 since for   $L^2_{\alpha}(\bbR)$
to be imbedded in $L^\infty$ we need $\alpha>1/2$. For $1<p<2$ the inequality can only hold if
$\alpha>\tilde \alpha(p)=d(\frac 1p-\frac 12)+\frac 12$. This is seen by
writing
\Be\label{Galpha}
G_\alpha f=  \Big(\int_0^\infty| K_t^\alpha* f|^2 \frac{dt}{t}\Big)^{1/2} \ \text{ where }\
\widehat {K_t^\alpha}(\xi) = \alpha\frac
{|\xi|^2}{t^2}\Big(1-\frac{|\xi|^2}{t^2}\Big)_+^{\alpha-1}.
\Ee
Then, for a  suitable Schwartz function  $ \eta$, with $\widehat \eta$  vanishing near $0$ and compactly supported in a
narrow  cone,  and for  $t\sim 1$ and large $x$ in an open cone,  we have
\Be \label{asympt}
 K_t^\alpha*\eta (x) =  c_\alpha t^d e^{it|x|} |tx|^{-\frac{d-1}{2}-\alpha} + E_t(x)\Ee
where $E_t$ are lower order error terms. This leads to
$$\Big(\int_1^2|K_t^\alpha*\eta|^2dt\Big)^{1/2} \in L^p(\bbR^d) \quad \implies
\alpha> \tilde \alpha(p)\,.
$$
Note that the oscillation for large $x$ in \eqref{asympt} plays no role here.

Concerning positive results for $p\le 2$, the  $L^2$-bound  for $\alpha>1/2$ follows from Plancherel and
was already observed in \cite{stein58}. The case
 $1<p\le 2$, $\alpha>\frac{d+1}2$ is covered by the Calder\'on--Zygmund theory for vector-valued singular integrals, and analytic interpolation yields $L^p$-boundedness for $1<p<2$, $\alpha> \tilde \alpha(p)$,
see \cite{su1}, \cite{igku}.
There is also an endpoint result for $\alpha=\tilde \alpha(p)$,
indeed one can use  the
arguments by Fefferman \cite{feff} for  the weak type endpoint inequalities
for  Stein's $g_\lambda^*$ function  to prove that $G_{\tilde\alpha(p)}$ is of weak type $(p,p)$ for $1<p<2$ (Henry Dappa, personal communication,
see also Sunouchi~\cite{su2} for the case  $d=1$).

The   range $2<p<\infty$ is more interesting, since now the oscillation of the kernel $K^\alpha_t$ plays a significant  role, and,
in  dimensions $d\ge 2$,
 the  problem is   closely related to  the Fourier restriction and
Bochner--Riesz problems.
A necessary condition
for $p>2$ can be obtained by duality.
Inequality \eqref{Galphaineq} for $p>2$
implies that
for all $b\in L^2([1,2])$ and $\eta$ as above
\Be\label{necessityb}
\Big\|\int_1^2 b(t)
K_t^\alpha*\eta \,dt  \Big\|_{p'}\lc \Big(\int_{[1,2]}|b(t)|^2 dt\Big)^{1/2}\,.
\Ee
If we again split $K_t^\alpha$ as in \eqref{asympt}, and prove suitable  upper bounds for the expression involving  the error terms then we see that, for  $R\gg 1$,
$$\int_{|x|\ge R} \Big|\frac{ \widehat b(|x|) }{|x|^{\frac{d-1}{2}+\alpha}}
\Big|^{p'} dx <
\infty,$$
which leads to the necessary condition  $\alpha> d(\frac 1{p'}-\frac 12)=d(\frac 12-\frac 1{p})$.

It is conjectured that   \eqref{Galphaineq}
holds for $2<p<\infty$ if and only if $\alpha>\alpha(p)=\max\{d(\frac 12-\frac 1p),\,\frac 12\}$.
For $d=1$ this can be shown in several ways, and the estimate follows from Calder\'on--Zygmund theory
(one such proof is in  \cite{su2}).
The full
 conjecture for $d=2$ was proved by Carbery \cite{carbery}, and
a variable coefficient generalization of his  result was later obtained in  \cite{mss2}.
The partial  result  for $p>\frac {2d+2}{d-1}$ which relies
on the  Stein--Tomas restriction  theorem  is  in Christ \cite{christ} and in
 \cite{seegerthesis}.
A better range (unifying the cases $d=2$ and $d\ge 3$) was recently obtained by the authors  \cite{lrs}; that is, 
inequality  \eqref{Galphaineq}  holds for $\alpha>d(1/2-1/p)$ and $d\ge 2$ in the range
$2+ 4/d<p<\infty$.
This  extends previous results
on  Bochner--Riesz means by the first  author \cite{lee} and relies on Tao's bilinear adjoint restriction theorem \cite{tao-bilinear}. 
Motivated by a still open problem of Stein \cite{st-williamstown}, the authors also 
proved a related weighted inequality  in \cite{lrs}, 
 namely for $d\ge 2$,   $1\le q<\frac{d+2}2$,
\[\int [G_\alpha f(x)]^2 w(x)\, dx \lc
\int |f(x)|^2 \fW_q w(x)\, dx, \quad \alpha>\frac{d}{2q} \,,
\]
where $\fW_q$ is an explicitly defined operator which is of weak type $(q,q)$ 
and  bounded on  $L^r$ with $q<r\le \infty$.
This is an analogue of a result by Carbery and the third author in two 
dimensions \cite{case} and extends a weighted inequality by Christ \cite{christ} in higher dimensions.
One might expect that recent progress by  Bourgain and 
Guth \cite{bogu} on the Bochner--Riesz problem will  
lead to further improvements in the ranges of these results   but this is currently open.

By the
equivalence
\eqref{ptwequiv}
one can interpret the boundedness of  $G_\alpha$  as a regularity result for 
spherical means and then for solutions of the wave equation. By 
a  somewhat finer analysis
in conjunction with the use of the Fefferman--Stein \#-function
\cite{fs-hardy}
the authors obtained 
 an  $L^p(L^2)$
endpoint result, local in time, in fact  not just for the wave
equation, but also for other dispersive  equations.  Namely 
if $\gamma>0$,
$d\ge 2$, $2+ 4/d <p<\infty$  then 
\Be\label{LpL2endpt}
\Big\|
\Big(\int_{-1}^1\big|  e^{it(-\Delta)^{\gamma/2}} \!f\big|^2 dt\Big)^{1/2}
\Big\|_p
\lc \|f\|_{B^p_{s,p}},  \quad
\frac{s}{\gamma}= d\Big(\frac 12-\frac 1p\Big)-\frac 12\,.
\Ee
Here  $B^p_{s,p}$ is the Besov space which strictly  contains the Sobolev space  $L^p_s$ for  $p>2$.

Concerning endpoint estimates, many  such  results for 
Bochner--Riesz multipliers and  variants 
had been  previously known (\cf. \cite{christ-rough}, \cite{christ-wtBR},
\cite{christ-sogge-inv}, \cite{se-ind}, \cite{se1}, \cite{taothesis}).
For the  Bochner--Riesz means $R^{\la}_t$ with  the critical exponent 
$\la(p)= d(1/2-1/p)-1/2$,
 Tao  \cite{taowt} showed  that if for 
some  $p_1>2d/(d-1)$ the    $L^{p_1}$ boundedness holds  for all 
$\la>\la(p_1)$,
then one also has  a bound in the limiting case, for $p_1<p<\infty$, 
namely 
  $R_t^{\la(p)}$ maps  $L^{p,1}$ to $L^p$, and $L^{p'}\to L^{p',\infty}$.
In contrast no positive  result for 
$G_{d/2-d/p}$ 
seems to have been  known, even for the  version with dilations restricted to $(1/2,2)$. 
It should be  emphasized  that, despite the pointwise equivalence 
of the two square functions  in \eqref{ptwequiv},
the sharp  regularity  result 
\eqref{LpL2endpt}
does not imply a corresponding endpoint bound for $G_{d/2-d/p}$
(in fact the latter is not bounded on $L^p$). 
 In this paper
we will prove  a sharp result  for $G_{d/2-d/p}$
in the restricted open range of the Stein--Tomas adjoint restriction theorem, and obtain 
 related results for 
maximal operators and  Fourier multipliers. 

\section{Endpoint results}\label{endptsect}

\begin{theorem}\label{sqthm} Let $d\ge 2$,
$\frac{2(d+1)}{d-1}<p<\infty$ and $\alpha= d(\frac 12-\frac 1p)$.
 Then
\Be\label{LpineqGalpha}\|G_\alpha f\|_p\le C\|f\|_{L^{p,2}}\, .\Ee
\end{theorem}
Here 
$L^{p,q}$ denotes the Lorentz space. 
We note that the  $L^p\to L^p$ boundedness fails; moreover  $L^{p,2}$ cannot be replaced by a larger space $L^{p,\nu}$ for $\nu>2$.
This can  be shown  by 
the argument in \eqref{necessityb}
namely, if $b\in L^2([1,2])$ then  the function $\widehat b(|\cdot|)
(1+|\cdot|)^{-\frac d{p'}+\frac 12}$  belongs to  $L^{p',2}$
but not necessarily to $L^{p',r}$ for $r<2$.
The space $L^{p,2}$ has occured earlier  in endpoint results 
related to other  square functions, see
 \cite{se-archiv}, \cite{setao}, \cite{taowright}.

The   pointwise bound  \eqref{maximaloperatorineq} and 
Theorem \ref{sqthm} 
yield a new bound for maximal functions, in particular for multipliers in the Sobolev space $L^2_{d/2-d/p}$ 
which are compactly supported away from the origin. 
This  Sobolev condition is too restrictive 
to give any endpoint bound for the maximal Bochner--Riesz operator. However such a result can be 
deduced from a related   result 
 on maximal functions
$$M_m f(x)= \sup_{t>0} |\cF^{-1}[ m(t|\cdot|) \widehat f\,](x)|
$$ with $m$  compactly supported away from the origin.
Our  assumptions 
involve the  Besov space 
$B^2_{\alpha,q}$ (which is 
$L^2_\alpha$ when $q=2$) and thus 
the following result  seems to be beyond the scope of a square function 
estimate  when $q\neq 2$.

\begin{theorem}\label{maxmult}
Let $d\ge 2$, $\frac{2(d+1)}{d-1}<p<\infty$, 
$\alpha=d(\frac 12-\frac 1p)$ and 
 $p'\le q\le \infty$. 
Assume that  $m$ is  supported 
in $(1/2,2)$ and that $m$  belongs to the Besov 
space $B^2_{\alpha,q}$.
Then 
$$\|M_m f\|_{L^p} \le C\|m\|_{B^2_{\alpha,q}}\|f\|_{L^{p,q'}}\,.
$$
\end{theorem}

We apply this to the Bochner--Riesz maximal operator $R^\la_*$ defined by
$$R^\la_* f(x)= \sup_{t>0}|R^\la_t f(x)|.$$
Split
$(1-t^2)_+^{\la}= u_\la(t)+ m_\la(t)$ where $m_\la$ is supported in $(1/2,2)$ and 
$u_\la\in C^\infty_0(\bbR)$. Then the maximal function  $M_{u_\la}f$ is pointwise
controlled dominated by the Hardy--Littlewood maximal function and thus 
bounded on $L^p$ for all $p>1$. The function 
$m_\la$ belongs to the Besov space $B^2_{\la+1/2, \infty}$ and   Theorem \ref{maxmult} with $q=\infty$ 
yields a  maximal version 
 of (the dual of) Christ's endpoint estimate in \cite{christ-wtBR}.

\begin{corollary} \label{rieszthm}
 Let $d\ge 2$, $\frac{2(d+1)}{d-1}<p<\infty$,  and 
$\la= d(\frac 12-\frac 1p)-\frac 12$. Then 
$$\|R^{\la}_*f\|_p\le C \|f\|_{L^{p,1}}\,.$$
\end{corollary}

\medskip

We now consider operators $T_m$ with   radial Fourier multipliers,  
as defined in \eqref{defofTm}, which do not necessarily decay at $\infty$.
The pointwise bounds \eqref{cgtpointwise},  Theorem \ref{sqthm} and duality 
yield optimal 
$L^p\to L^{p,2}$  estimates in the range 
$1<p<\frac{2(d+1)}{d+3}$,
for H\"ormander type 
multipliers with localized $L^2_\alpha$ conditions in the critical  case
$\alpha=d(\frac 1p-\frac 12)$.
This  
demonstrates the effectiveness of Stein's point of view 
in~\eqref{steinptw} and \eqref{cgtpointwise}.

The following more general theorem  is again 
beyond the scope of a square function estimate.
We   use dilation invariant assumptions 
involving  localizations of Besov spaces $B^2_{\alpha,q}$.
We note that in \cite{se-ind} it had been 
left open  whether one could use  endpoint 
Sobolev space or  Besov spaces with $q>1$  in 
 \eqref{besovassumpt} below.


\begin{theorem} \label{fmthm}
Let $d\ge 2$,
$1<p<\frac{2(d+1)}{d+3}$,  $\alpha=d(\frac 1p-\frac 12)$
and  $p\le q\le \infty$. Let 
$\varphi_\circ$ be a nontrivial $C^\infty_0$ function supported in $(1,2)$.
Assume
\Be\label{besovassumpt}
\sup_{t>0}\|\varphi_\circ m(t\,\cdot\,)\|_{B^2_{\alpha,q}} <\infty\,.
\Ee
Then $T_m$ maps $L^p$ to $L^{p,q}$ and $L^{p',q'}$ to $L^{p'}$.
\end{theorem}

It is not hard to see that the assumption
\eqref{besovassumpt} is independent of the choice of the particular cutoff
$\varphi_\circ$.
The result is sharp as $T_m$ does not map $L^p$ to $L^{p,r}$ for $r<q$.
This can be seen by considering  some test multipliers of Bochner--Riesz
type. Indeed, let $\Phi_1$ be a radial $C^\infty$ function, with $\Phi_1(x)=1$
for $2^{-1/2}\le |x|\le 2^{1/2}$ and supported in $\{1/2< |x|<2\}$
and similarly let $\chi$ be a radial $C^\infty$ function compactly supported away from the origin and so that $\chi(\xi)=1$ in a neighborhood of the unit sphere.
Set (now with $p<2$)
$$m(\xi)=\chi(\xi) \sum_{j=1}^\infty c_j
\int(1-|\xi-\eta|^2)_+^{d(\frac 1p-\frac 12)-\frac 12} 2^{jd}\widehat\Phi_1(2^j\eta) \, d\eta\,.$$
We first remark that if we write
$m(\xi)=m_\circ(|\xi|)$, then $m_\circ\in B^2_{\alpha,q}(\bbR) $ if and only if
$m\in B^2_{\alpha,q}(\bbR^d)$
(here we use that $m_\circ $ is compactly supported  away from the origin).
Now considering the explicit formula for the kernel of Bochner--Riesz means
(\cf. \eqref{besselalpha} below) it is easy to see
 that $m\in B^2_{d/p-d/2,q}(\bbR^d)$ if and only if
$\{c_j\}_{j=1}^\infty $ belongs to $\ell^q$; moreover the necessary condition 
$\cF^{-1} [m]\in L^{p,q}$ is satisfied  if and only if $\{c_j\}$ belongs to~$\ell^q$.
These considerations show  the sharpness of Theorem~\ref{fmthm}  and also the sharpness of Theorem \ref{maxmult}.

For the operator $T_m$ acting on the subspace $L^p_\rad$, consisting of
  radial $L^p$ functions, the  estimate corresponding to
Theorem \ref{fmthm}  has been known to be true in the optimal range $1<p<\frac{2d}{d+1}$. In fact  Garrig\'os and the third author \cite{gs} obtained an actual
characterization of classes of Hankel multipliers which yields, for
$p\le q\le\infty$,
$$\|T_m\|_{L_\rad^{p}\to L^{p,q}} \approx
\sup_{t>0}\big\|\cF^{-1}\big[\phi(|\cdot|) m(t|\cdot|)\big]\big\|_{L^{p,q}(\bbR^d)}\,\text{ if $1<p<\frac{2d}{d+1}$. }$$
This easily implies the $L^p_\rad\to L^{p,q}$ boundedness under assumption  \eqref{besovassumpt}, see \cite{gs}.
Similarly,  if in
Theorem \ref{fmthm} we replace the range $(1, \frac{2d+2}{d+3})$ with  the smaller $p$-range
$(1, \frac{2d-2}{d+1}) $ (applicable only in dimension  $d\ge 4$) the result
follows from the   characterization of radial $L^p$ Fourier multipliers  acting on general $L^p$ functions
 in a recent article by Heo, Nazarov and the third author \cite{hns}.
There it is proved that for $p\le q\le \infty$,
\Be\label{hnsest}
\|T_m\|_{L^{p}\to L^{p,q}} \approx \sup_{t>0}\big\|\cF^{-1}\big[\phi(|\cdot|) m(t|\cdot|)\big]\big\|_{L^{p,q}(\bbR^d)}
\, \text{ if $1<p<\frac{2d-2}{d+1}$}\,.\Ee

The remainder of this paper is devoted to the proofs 
of the above theorems. They are mostly   based on  ideas in \cite{hns}. 
It remains an interesting open problem
 to extend the range of  \eqref{hnsest}, in particular to prove such a result for some $p>1$ 
 in  dimensions two and three. 
Moreover it would be interesting to prove the above theorems 
beyond  the Stein--Tomas range.

\section{Convolution with spherical measures}\label{convsphmeas}
In this section we prove
an inequality for convolutions with spherical measures acting on functions with
a large amount of cancellation.
It  can be used to obtain  results such as Theorems  \ref{fmthm} for
radial multipliers which are compactly supported away from the origin.

To formulate this inequality let $\eta$ be a Schwartz function on $\bbR^d$ and let
$\psi$ be a radial
$C^\infty$ function with compact support in $\{x:|x|\le 1\}$ and
such that $$\widehat \psi(\xi)= u(|\xi|)$$ vanishes of order $10d$ at the origin.
For $j\ge 1$ let $I_j=[2^j, 2^{j+1}]$ and  denote by
$\sigma_r$  the surface
measure on the sphere of radius $r$ which is centered at the origin.
Thus the norm of $\sigma_r$ as a measure is $O(r^{d-1})$. We  recall
 the Bessel function formula
\Be\label{besselfct}\widehat{\sigma}_r(\xi)= r^{d-1} \Jscr(r|\xi|) \text{ with } \Jscr(s)= c(d)  s^{-\frac{d-2}{2}}J_{\frac{d-2}{2}}(s)\,,
\Ee
which implies  $|\widehat{\sigma}_r(\xi)|\lc r^{d-1} (1+r|\xi|)^{-\frac{d-1}{2}}$. In view of the assumed cancellation of $\psi$, we have
\Be\label{Fsigmar}\|\widehat{\psi*\sigma_r}\|_\infty = O(r^{(d-1)/2}).\Ee

In what follows let $\nu$ be a probability measure on $[1,2]$. 
We will need to 
 work with  functions with values in the Hilbert space
$\cH=L^2(\bbR_+,\frac{dr}r)$ and write
$$\|F\|_{L^p(L^1(\cH))}
= \Big\|\int_1^2\Big(\int_0^\infty|F_t(r,\cdot)|^2\frac{dr}{r}\Big)^{1/2}\, d\nu(t)
\Big\|_{p}.$$

\begin{proposition}\label{main2}
Let $1\le p< \frac{2(d+1)}{d+3}$. Then
$$
\Big\|\sum_{j\ge 1} \int_1^2\int_{I_j} \psi*\sigma_{rt}* \eta* F_{t,j}(r,\cdot) \,
dr\,d\nu(t)\Big\|_p \lc\Big(\sum_{j\ge 1} 2^{jd}
\big\|F_{j}\big\|_{L^p(L^1(\cH))}^p
\Big)^{1/p}.
$$
\end{proposition}
The measure $\nu$ is used here to unify the proofs of Theorems \ref{maxmult} and \ref{fmthm}.
For our applications we  are only 
interested  in two such measures. For Theorems 
\ref{sqthm} and  \ref{fmthm} we take for $\nu$ the Dirac measure at $t=1$
(and consequently in this case we can  set $\sigma_{rt}=\sigma_r$ and eliminate all $t$-integrals in the proofs below). For the application to Theorem \ref{maxmult} we take for $\nu$ the  Lebesgue measure on $[1,2]$.

We first give a proof for the $L^p$ bound of each  term in the $j$-sum,
which uses standard arguments (\cite{feff}, \cite{feff-isr}).

\begin{lemma} \label{STlemma}
Let $1\le p\le \frac{2(d+1)}{d+3}$. Then
$$
\Big\|\int_1^2\int_{I_j} \psi*\sigma_{rt}* F_t(r,\cdot) \,
dr\,d\nu(t)\Big\|_2\lc 2^{jd/2}
\big\|F\|_{L^p(L^1(\cH))}\,.
$$
\end{lemma}

\begin{proof}
We use Plancherel's theorem and then the Stein--Tomas restriction theorem
\cite{tomas}. With  $\Jscr$ as in \eqref{besselfct}
so that
 $\widehat \sigma_r(\xi)= r^{d-1} \Jscr(r|\xi|)$ and $\widehat \psi(\xi)= u(|\xi|)$, we get from  the restriction theorem
 \begin{align*}
 &\Big\|\int_1^2
\int_{I_j}\psi*\sigma_{rt}*
F_t(r,\cdot) \,dr\, d\nu(t)
\Big\|_2^2
\\&=c
 \int|u(\rho)|^2
\int_{S^{d-1}}\Big|\int_1^2\int_{I_j} (rt)^{d-1}\Jscr(rt\rho)
\widehat{F_t}(r,\rho\xi') \, dr
\,d\nu(t)
\Big|^2 d\sigma(\xi') \, \rho^{d-1}d\rho
\\
&\lc
 \int|u(\rho)|^2\rho^{\frac{2d}{p}-d-1}
\Big\|\int_1^2
\int_{I_j} (rt)^{d-1}\Jscr(rt\rho) F_t(r,\cdot) \,
dr\,d\nu(t)\Big\|_p^2  \, d\rho
\\
&\lc
 \Big\|
\int_1^2
\Big(\int|u(\rho)|^2\rho^{\frac{2d}{p}-d-1}
\Big|     \int_{I_j} r^{d-1}\Jscr(rt\rho) F_t(r,\cdot) \,
dr\Big|^2d\rho\Big)^{1/2}
d\nu(t)
\Big\|_p^2\,.
\end{align*}
In the last step we have used Minkowski's integral inequality.
We claim that, for fixed $x\in \bbR^d$ and  $t\in [1,2]$,
\Be\label{fixedxest}
\int|u(\rho)|^2\rho^{\frac{2d}{p}-d-1} \Big|     \int_{I_j}
r^{d-1}\Jscr(rt\rho) F_t(r,x) \, dr\Big|^2 d\rho \lc \int_{I_j} \big|
F_t(r,x) \big|^2 \,r^{d-1} dr\,,
\Ee
with the implicit constant uniform in $x,t$,
and the lemma follows  by substituting this
in the previous display.

To see  \eqref{fixedxest} we first notice 
that for a radial  $H(w)=H_\circ(|w|)$ we have
$$\int H_\circ(r) r^{d-1} \Jscr(r|\xi|)\,dr =c_d \widehat H(\xi).
$$
Thus, if we take $H^{x,t}(w)=\chi_{I_j}(|w|)F_t(|w|,x)$, the left-hand side of \eqref{fixedxest}  is a constant multiple of
\begin{align*}
&\int|\widehat \psi(\xi)|^2 |\xi|^{\frac{2d}{p}-2d} |\widehat
{H^{x,t}}(t\xi)|^2 d\xi
\\& \lc \int |\widehat {H^{x,t}}(\xi)|^2 d\xi = c \int |H^{x,t}(w)|^2 dw = c'
\int_{I_j} |F_t(r,x)|^2 r^{d-1}\,dr\,,
\end{align*}
and we are done.
In the inequality we used that $\widehat\psi$ vanishes of high order
at the origin.
\end{proof}
If we fix $j$  and assume that $F_{Q,t}(r,\cdot)$ is supported for all $r$ in a
cube $Q$  of sidelength $2^j$ then  the expression
$\int_1^2\int_{I_j} \psi*\sigma_{rt}* F_{Q,t}(r,\cdot) dr\,d\nu(t)$ is supported in a similar slightly larger  cube.
From this it quickly follows   that
\begin{multline*}
\Big\|\int_1^2\int_{I_j} \psi*\sigma_{rt}* F_{Q,t}(r,\cdot) \,
dr\,d\nu(t)\Big\|_p\\\lc 2^{jd/p}
\Big\|\int_1^2\Big(\int_{0}^\infty|F_{Q,t}(r,\cdot)|^2\frac{dr}{r}\Big)^{1/2}d\nu(t)\Big\|_p.
\end{multline*}
This estimate is however insufficient to prove Proposition \ref{main2} for
$p>1$.
We shall also  need the following  orthogonality lemma.
\begin{lemma} \label{orthlemma}
Let $J_1, J_2\subset (0,\infty)$ be intervals and let $E_1$, $E_2$ be compact sets
in $\bbR^d$  with $\dist(E_1,E_2)\ge M\ge 1$. Suppose that for every $r\in J_i$,
the function $x\mapsto f_{i}(r,x)$ is supported in $E_i$. Then,
for $t_1,t_2 \in [1,2]$,
\begin{multline*}
\Big|\int_{J_1}\int_{J_2} \biginn{\psi*
\sigma_{r_1t_1} * f_{1}(r_1,\cdot)}
{\psi* \sigma_{r_2t_2} * f_{2}(r_2,\cdot)} dr_1dr_2\, 
\Big| \\
\lc M^{-\frac{d-1}{2}}
 \prod_{i=1}^2\Big[\int_1^2\int \Big(\int_{J_i}|f_{i}(r, y)|^2r^{d-1}
 dr\Big)^{1/2}  dy\Big]\,.
\end{multline*}
\end{lemma}

\begin{proof}
We follow  \cite{hns} and
 apply Parseval's identity  and  polar coordinates in~$\xi$.
Then, 
\begin{align*}
&\biginn {\psi* \sigma_{r_1t_1}* f_{1}(r_1, \cdot)}{\psi* \sigma_{{r_2t_2}} *
  f_{2}(r_2,\cdot)}
\\&=c\int |\widehat\psi(\xi)|^2
\widehat \sigma_{r_1t_1}(\xi)  \overline{\widehat \sigma_{r_2t_2}(\xi) }
\iint f_{1}(r_1,y_1)  \overline{f_{2}(r_2,y_2) } e^{i\inn{\xi}{y_2-y_1}}\, 
dy_1dy_2d\xi
\\
&=c'\int |u(\rho)|^2
(r_1t_1)^{d-1} \Jscr(r_1t_1\rho) (r_2t_2)^{d-1} \Jscr(r_2t_2\rho)\,
\times\\&\quad\quad\quad\quad\quad\quad\quad\iint f_{1}(r_1,y_1)
\overline{f_{2}(r_2,y_2) } \Jscr(\rho|y_1-y_2|)
 \,dy_1\,dy_2\,\rho^{d-1}d\rho,
\end{align*}
so that the left-hand side of the desired inequality
 is equal to a constant multiple of
\begin{multline}\label{polarexpression}
\int\iint  |u(\rho)|^2 
\int_{J_1} (r_1t_1)^{d-1}\Jscr(r_1t_1\rho)f_{1}(r_1,y_1)
\, dr_1\,
\\
\times\,\int_{J_2} (r_2t_2)^{d-1}\Jscr(r_2t_2\rho)\overline{f_{2}(r_2,y_2)}
\, dr_2\,
 \,\,\Jscr(\rho|y_1-y_2|)
\,dy_1\,dy_2\,\rho^{d-1}d\rho\,.
\end{multline}
Now define two radial kernels by
$H^{y_i}_{i}(w)
= f_{i}(|w|,y_i) \chi_{J_i}(|w|)$
so  that the expression \eqref{polarexpression} can be written  as a constant times
\Be \int  \iint |\widehat\psi(\xi)|^2 \widehat H^{y_1}_{1}(t_1\xi)
\overline{\widehat H^{y_2}_{2}(t_2\xi)}\,\Jscr(|\xi||y_1-y_2|) \, dy_1
dy_2\,d\xi. \Ee Then, using the decay for Bessel functions and the $M$-separation assumption,
$$|\Jscr(|\xi||y_1-y_2|)|
\lc (1+\rho M)^{-\frac{d-1}{2}},\quad  y_i\in E_i,\ \,i=1,2.
$$
By   the Cauchy--Schwarz inequality,
the left-hand side of the desired inequality is thus bounded by
\begin{align*}
&\prod_{i=1,2}\Big[\int_{y_i\in \bbR^d}
\Big(\int \frac{|\widehat \psi(\xi)|^2}
{(1+|\xi|M)^{\frac{d-1}{2}}}
|\widehat {H^{y_i}_i}(t_i\xi)|^2 d\xi\Big)^{1/2} dy\Big]
\\&\lc M^{-\frac{d-1}2} \prod_{i=1,2}\Big[ \int_{y\in \bbR^d} 
\big\|\widehat{ H^{y_i}_i}\big\|_2 dy\Big]\,,
\end{align*} and by Plancherel's theorem  this is
\begin{align*}
\\&\lc M^{-\frac{d-1}{2}} \prod_{i=1,2}
\Big[ \int
\Big(\int_{w_i\in \bbR^d}|f_{i}(|w|,y)|^2 \chi_{J_i}(|w|) dw\Big)^{1/2} 
dy\Big]
\\&\lc M^{-\frac{d-1}{2}}
\prod_{i=1,2} \Big[\int
\Big(\int_{J_i}|f_{i}(r,y)|^2  r^{d-1}\,dr\Big)^{1/2} dy\Big]\,,
\end{align*}
and so we are done.
\end{proof}

\subsubsection*{Proof of Proposition \ref{main2}.}
The case $p=1$ is trivial and  we assume $p>1$
in what follows.
For $z\in \bbZ^d$ consider the cube $q_z$ of all $x$ with $z_i\le x_i< z_i+1$ for $i=1,\dots, d$. Let
$$\gamma_{j,z}(f)
=  \sup_{x\in q_z} \int_1^2\Big(\int_0^\infty \Big|\int \eta(x-y) F_{j,t}(r,y)\, dy
\Big|^2\frac{dr}{r}\Big)^{1/2}d\nu(t),$$
and since $\eta$ is a Schwartz function it is straightforward to verify that,  for every $j$, 
\Be\label{gammadiscr}
\Big(\sum_{z\in \bbZ^d}|\gamma_{j,z}(f)
|^p\Big)^{1/p}\lc
\Big\|\int_1^2\Big(\int_0^\infty |F_{j,t}(r,\cdot)|^2\frac{dr}{r}\Big)^{1/2}d\nu(t)
\Big\|_p\,,
\Ee
with the implicit constant independent of $j$.
If
$\gamma_{j,z}(f)\neq 0$ we set
$$b_{j,z,t}(r,x)=
 [\gamma_{j,z}(f)]^{-1} \chi_{q_z}(x)\int \eta(x-y) F_{j,t}(r,y)\, dy
$$ and if
$\gamma_{j,z}(f)=0$  we set $b_{j,z,t}=0$. Then
\Be \label{bjzassumpt}
\sum_{z\in \bbZ^d}  \sup_{x\in q_z}\int_1^2
 \Big(\int_0^\infty|b_{j,z,t}(r,x)|^2 \frac {dr}{r}\Big)^{1/2}d\nu(t)\le 1.
\Ee
Let
$$V_{j,z}(x)= \int_1^2\int_{I_j} \psi*\sigma_{rt}* b_{j,z,t}(r,x)\, dr\,d\nu(t). $$
In view of \eqref{gammadiscr} it suffices to show that
for arbitrary functions $z\mapsto \gamma_{j,z}$ on $\bbZ^d$ we have,
 for $1<p<\frac{2(d+1)}{d+3}$,
\Be\label{Lponlattice}\Big\| \sum_{j\ge 1}\sum_{z\in \bbZ^d}\gamma_{j,z}
 V_{j,z}
\Big\|_p \lc\Big(\sum_{j\ge 1} \sum_{z\in \bbZ^d}  |\gamma_{j,z}|^p 2^{jd}
 \Big)^{1/p}\,
\Ee
where the implicit constant is independent of the specific choices of the
$b_{j,z,t}$ (satisfying \eqref{bjzassumpt} with $b_{j,z,t}$ supported in $q_z$).
Let $\mu_d$ denote the measure on $\bbN\times\bbZ^d$ given by
$$\mu_d(E)= \sum_{j\ge 1} 2^{jd} \#\{z\in\bbZ^d: (j,z)\in E\}\,.$$
Then \eqref{Lponlattice} expresses  the
$L^p(\bbZ^d\times \bbN,\mu_d)\to L^p(\bbR^d)$ boundedness of an operator $\cT$. In the open $p$-range it suffices by real interpolation to show that
$\cT$ maps $L^{p,1}(\bbZ^d\times \bbN,\mu_d)$ to $L^{p,\infty}(\bbR^d)$.
 This amounts to checking the restricted weak-type inequality
\Be\label{restronlattice}
\meas\big(\big\{x:  \big|\sum_{j\ge 1}\sum_{z\in \cE_j} V_{j,z}
\big|\,>\,\la\big\}\big) \lc \la^{-p} \sum_{j\ge 1}  2^{jd} \#(\cE_j)
\Ee
where  $\cE_j$ are finite subsets of $\bbZ^d$.
Now for each $(j,z)$ the
term   $V_{j,z}$
is supported on a ball
of radius $C2^{j+1}$ and therefore  the entire sum is supported on a set of measure $\ls \sum_{j\ge 1}  2^{jd} \#(\cE_j)$. Thus the estimate \eqref{restronlattice}
holds for $\la\le 10$. Assume now that $\la>10$.

We decompose $\bbR^d$
into dyadic \lq half open\rq  \, cubes of
sidelength $2^j$ and let~$\fQ_j$
be the collection of these
$2^j$-cubes. For each $Q\in \fQ_j$ let
$Q^*$ be  the cube  with same center as $Q$ but sidelength $2^{j+5}$.
Note that for $z\in Q$ the term $V_{j,z}$
is supported in $Q^*$. Letting
$$\fQ_j(\la)\,:=\{Q\in\fQ_j\,:\,
\#(\cE_j\cap Q)> \la^p\,\}$$ and
$$\Omega= \bigcup_j \bigcup_{Q\in \fQ_j(\la)} Q^*\,,$$
we have the favorable estimate
\begin{align*}
\meas(\Omega) &\le 2^5 \sum_{j\ge1} \sum_{Q\in \fQ_j(\la)}|Q|
\le 2^5 \sum_{j\ge1} 2^{jd}\sum_{Q\in \fQ_j(\la)}\frac{\#(\cE_j\cap Q)}{\la^p}
\notag \\
&\lc \la^{-p} \sum_{j\ge1}2^{jd}\#(\cE_j)\,.
\label{excset}
\end{align*}
Thus the remaining  estimates need  only involve the \lq good\rq \ part of
$\cE_j$;
$$\cE_j^{\la}= \bigcup_{Q\in \fQ_j\setminus\fQ_j(\la)} Q\cap \cE_j.$$ Note that every subset of diameter $C 2^j$, with  $C>1$, contains  $\lc C^d\la^p$
points in $\cE_j^\la$. Letting
$$V_j = \sum_{z\in \E_j^\la} V_{j,z}\,,$$
it remains  to show that
$$
\meas\big(\big\{x:  \big|\sum_{j\ge 1} V_j(x)\big|\,>\,\la\big\}\big) \lc \la^{-p} \sum_{j\ge 1}  2^{jd} \#(\cE_j)\,.
$$
This will follow from
\Be\label{mainL2}
\Big\|\sum_{j\ge1} V_j\Big\|_2^2 \le C \la^{\frac{2p}{d+1} }\log \la
\sum_{j\ge 1}  2^{jd} \#(\cE_j)
\Ee
and Tshebyshev's inequality since, for $p<\frac{2(d+1)}{d+3}$ and $\la>1$,
$$\la^{\frac{2p}{d+1} -2 }\log \la  \le C_p \la^{-p}.$$

\medskip

\noi{\it Proof of \eqref{mainL2}.}
Setting $ N(\la)= 10 \log_2\la$,  we treat the sums over $j\le N(\la)$ and $j> N(\la)$ separately.
Using the Cauchy--Schwarz inequality for the first sum, \begin{multline} \label{Vjexp}\Big\|\sum_{j\ge 1}V_j\Big\|_2^2
\,\lc\, \log (\la) \sum_{j\le N(\la)}  \|V_j\|_2^2
+ \sum_{j> N(\la)}  \|V_j\|_2^2\\
+ \sum_{j> N(\la)}  \sum_{ N(\la )< k< j-10}
\big|\biginn{V_j}{V_k}\big|.
\end{multline}
Since the expression
$\sum_{z\in \cE_j^\la\cap Q} V_{j,z}$  is supported in $Q^*$ it follows
easily from Lemma \ref{STlemma} (applied with
the endpoint exponent $\frac{2(d+1)}{d+3}$)  that
\begin{align*}
\|V_j\|_2^2 &\lc \sum_{Q\in \fQ_j}
\Big\|\int_1^2
 \Big(\int_{I_j} 2^{jd} \Big|\sum_{z\in Q\cap \cE_j^\la} b_{j,z,t}(r,\cdot) \Big|^2 \, \frac{dr}{r}\Big)^{1/2}d\nu(t)\Big\|_{\frac{2(d+1)}{d+3}}^2.
\end{align*}
Since $Q\cap \cE_j^\la$ contains no more
than $\la^p$ points we have by \eqref{bjzassumpt}
\begin{multline*}
\Big\|
 \int_1^2\Big(\int_{I_j} 2^{jd} \Big|
\sum_{z\in Q\cap \cE_j^\la} b_{j,z,t}(r,\cdot) \Big|^2 \, \frac{dr}{r}\Big)^{1/2}d\nu(t)\Big\|_{\frac{2(d+1)}{d+3}}^2\
\\
\lc 2^{jd}
\big(\#(\cE_j^\la\cap Q)\big)^{\frac {d+3}{d+1}} \lc  2^{jd} \#(\cE_j\cap Q)
\la^{p \frac{2}{d+1}}
\end{multline*}
and thus
\Be\label{squaresumVj}
\sum_{j=1}^\infty \|V_j\|_2^2 \lc \la^{p \frac{2}{d+1}} \sum_{j} 2^{jd} \#\cE_j\,.
\Ee
Thus  we get the asserted bound \eqref{mainL2}
for the sum of the first two terms on the right-hand
side of $\eqref{Vjexp}$.

It remains to estimate the mixed terms
$\inn{V_j}{V_k}$ for $N(\la)<k<j-10$.
For fixed $j,k$ we
let $I_{j,k}^n=[2^kn, 2^k(n+1)]\cap I_j$ with $n\in \bbZ$, $n\approx 2^{j-k}$.
Then with
\begin{align*}
V^{k,n}_{j,z,t}:&= \psi* \int_{I^n_{j,k} } \sigma_{rt}*b_{j,z,t}(r,\cdot)\,dr
\\
V_{k,z'}:&=  \psi* \int_1^2\int_{I_k} \sigma_{rs}*b_{k,z',s}(r,\cdot)\,dr\,d\nu(s)
\end{align*}
we can write
$$
\inn{V_j}{V_k}=\\ \int_1^2\sum_n \sum_{z\in \cE_j^\la}
\sum_{z'\in \fZ_k(n,z,t)} \inn{V^{k,n}_{j,z,t}}{V_{k,z'}}\,d\nu(t);
$$
here, in view of the support properties, we were able to restrict 
the $z'$ summation to the set
$$\fZ_k(n,z,t):=\{z'\in \cE_k^\la: \big| |z'-z| -n t 2^k\big|\le C 2^k\},$$
with $C$ a suitable constant.
Observe that for $z'\in \fZ_k(n,z,t)$, with $k\le j-10$, we have $|z-z'|\approx 2^j$ since $n t2^k\in I_j$.

By Lemma \ref{orthlemma} (applied with the parameter $M\approx 2^j$)
we have for fixed $z,z',t$,
\Be\label{orthlemmaappl}
\big|\inn{V^{k,n}_{j,z,t}}{V_{k,z'}}|
\big| \lc 2^{-j\frac{d-1}{2}}
\int_{|y-z|\le C}
h_{j,k,t}^{z,n}(y)\, dy \,\int_{|y'-z'|\le C} \,\h_k^{z'}(y')\, dy'
\Ee
with
\begin{align*}
h_{j,k,t}^{z,n}(y) &=
\Big(\int_{I_{j,k}^n}|b_{j,z,t}(r,y)|^2 r^{d-1} dr\Big)^{1/2}\,,\\
h_k^{z'}(y)&=
\int_1^2\Big(\int_{I_k}|b_{k,z',s}(r,y)|^2 r^{d-1} dr\Big)^{1/2}d\nu(s) \,.
\end{align*}
By our normalization assumption \eqref{bjzassumpt},
\Be \label{hfunctionbounds}
\int_1^2\Big(
\sum_n
|h_{j,k,t}^{z,n}(y) |^2\Big)^{1/2}d\nu(t)
\lc 2^{\frac{jd}{2}}
\qquad \text{and}\qquad
h_k^{z'}(y')\,\lc 2^{\frac{kd}{2}}\,
\Ee
and, by the Cauchy--Schwarz inequality, we also have
\Be\label{CSforhnu}
\int_1^2 \sum_n
|h_{j,k,t}^{z,n}(y) | d\nu(t) \lc 2^{\frac{jd}{2}}
2^{\frac{j-k}{2}}.\Ee
Altogether, using \eqref{orthlemmaappl}   and \eqref{hfunctionbounds},
\begin{multline*}\label{GjGksec}
|\inn{V_j}{V_k}|\,\\\lc 2^{-j\frac{d-1}{2}}
\sum_{z\in \cE_j^\la}\sum_n\int_1^2
\int_{|y-z|\le C}
h_{j,k,t}^{z,n}(y)
\, dy\, 2^{kd/2} \#(\fZ_k(n,z,t))\,d\nu(t)\,.
\end{multline*}
Recall  that for every cube $Q$ of sidelength $2^k$ the set $\fZ_k(n,z,t)\cap Q$ contains at most $\la^p$ points.
Moreover, for each $z,n,t$ there are no more than
$O(2^{(j-k)(d-1)})$ dyadic cubes of sidelength~$2^k$ which intersect
$\fZ_k(z,n,t)$.
Thus \begin{equation*}\label{fZest}
\#(\fZ_k(n,z,t))
\lc \la^p 2^{(j-k)(d-1)}\,.
\end{equation*} This and
\eqref{CSforhnu}  yield,  for $k\le j-10$,
\begin{align*}\label{GjGksec}
|\inn{V_j}{V_k}|\,&\lc  2^{-j\frac{d-1}{2}}
\sum_{z\in \cE_j^\la}
\int_{|y-z|\le C}
\int_1^2\sum_{n}
h_{j,k,t}^{z,n}(y) d\nu(t)
\, dy\, 2^{kd/2} \la^p 2^{(j-k)(d-1)}\,
\\&\lc
2^{-j\frac{d-1}{2}}
\#(\cE_j^\la)
2^{\frac{jd}{2}}
2^{\frac{j-k}{2}} 2^{\frac{kd}{2}} \la^p 2^{(j-k)(d-1)}\,
\lc \la^{-p} 2^{-k\frac{d-1}{2}}2^{jd} \#(\cE_j^\la)\,.
\end{align*}
By summing a geometric series, we see that
$$\sum_{j> N(\la)}  \sum_{ N(\la)< k< j-10}
\big|\biginn{V_j}{V_k}\big| \lc  \la^p  2^{-N(\la)\frac{d-1}{2}}
\sum_{j\ge 1} 2^{jd} \#\cE_j,
$$
and by  the choice of $N(\la) =10 \log_2\la$, we have
$\la^p 2^{-N(\la)\frac{d-1}{2}}\lc \la^{p-5}\lc 1$. This
gives the desired estimate (indeed a better estimate)  for the
third term on the right-hand side of \eqref{Vjexp} and
 finishes the proof of \eqref{mainL2}.\qed

\subsection*{Lorentz space estimates}
We will use the following interpolation lemma 
in which we allow any $d>0$; this is the only place  where $d$
does not necessarily denote the dimension.

\begin{lemma} \label{weightedinterpol}
Let  $1\le p_0<p_1$, $d>0$,  and, for $j\in \bbN$, let  $S_j$ be an operator acting on functions on a measure
space  $(\cM,\mu)$ with values in a Banach space $\cB$. Suppose that  the inequality
\Be\label{strongassumption}
\Big\|\sum_{j\ge 1} S_j g_j\Big\|_{p_i} \le M_i
\Big(\sum_{j\ge 1} 2^{jd}\big\| g_j\big\|_{L^{p_i}(\cB)}^{p_i}
\Big)^{1/p_i}
\Ee
holds for $i=0,1$. Then for $p_0<p<p_1$, $\frac 1p=\frac{1-\vartheta}{p_0}+\frac{\vartheta}{p_1}$, and $p\le q\le \infty$,
$$
\Big\|\sum_{j\ge 1} 2^{-jd/p} S_j f_j \Big\|_{L^{p,q}} \le C_{p,q} M_0^{1-\theta}M_1^\theta
 \Big\|\Big(\sum_{j\ge 1}|f_j|_\cB^q\Big)^{1/q}\Big\|_{L^p}
$$
with $q=\infty$ interpreted as usual by taking a supremum.
\end{lemma}

\begin{proof}
Let $\mu^d$ denote the measure on $\bbN\times \cM$ given by
$$\mu^d(E)= \sum_{j\ge 1} 2^{jd} \int_{x:(j,x)\in E} d\mu\,.$$
By real interpolation of the assumptions \eqref{strongassumption}
we have
 \Be\label{maininterpol2} \Big\|\sum_{j\ge 1} S_jg_j\Big\|_{L^{p,q}} \le
C_{p,q} M_0^{1-\theta}M_1^\theta
\big\|\{g_j\}
\big\|_{L^{p,q}(\mu^d,\cB)} \,.
\Ee
We may apply this with $g_j= 2^{-jd/p} f_j$ and then our assertion  follows
 from the inequality
\Be\label{LorentzversusLq}
\big\|\{2^{-j d/p} f_j\}
\big\|_{L^{p,q}(\mu^d,\cB)}  \le \big\|\{f_j\}\big\|_{L^p(\ell^q(\cB))}\,.
\Ee
 The case for $p=q$ is  immediate. We also have
\begin{align*}
\mu^d\big(\big\{(j,x):  2^{-\frac{j d}{p}}
|f_j(x)|_\cB>\la\big\}\big) \le \mu^d\big(\big\{(j,x)\,:\,
2^{-\frac{j d}{p}} \sup_k|f_k(x)|_\cB>\la\big\}\big)&
\\
\qquad\qquad= \int \sum_{\substack {j\,:\, 2^{jd}<\\ \sup_k|f_k(x)|_\cB^p
\la^{-p}}} 2^{jd}\, dx
\,\le \,\la^{-p}\int \sup_k|f_k(x)|_\cB^p \, dx,&
\end{align*}
which yields \eqref{LorentzversusLq} for $q=\infty$.
By complex  interpolation (with fixed $p$)  we obtain
 \eqref{LorentzversusLq} for $p\le q\le\infty$.
\end{proof}

As an immediate consequence of
Lemma \ref{weightedinterpol} we obtain
a Lorentz space version of  Proposition \ref{main2} which is the main ingredient in the proof 
of Theorem \ref{maxmult}.
\begin{corollary}\label{mainlor2}
Let $1<p<\frac{2(d+1)}{d+3}$ and $p\le q\le \infty$. Then
\begin{multline*}
\Big\|\sum_{j\ge 1}2^{-\frac{j d}{p}}\int_1^2\int_{I_j}
 \psi*\eta*
\sigma_{rt}*
F_{j,t}(r,\cdot) \, dr\,d\nu(t)
\Big\|_{L^{p,q}} \\
\lc \Big\|\int_1^2\Big(\sum_{j\ge 1}
|F_{j,t}|_{\cH}^q\Big)^{1/q}d\nu(t)\Big\|_p\,.
\end{multline*}
\end{corollary}

\subsubsection*{A preparatory result} For the proof of Theorems \ref{sqthm} and  \ref{fmthm}
we shall need a more technical variant of the
corollary which is compatible  with atomic decompositions. 
In what follows we let $\nu$ be Dirac measure at $t=1$ so that the integrals in $t$ disappear.
Let $\ell\ge 1$ and for $\fz\in \bbZ^d$ let
$$R^{\ell}_{\fz}=\{x:2^{\ell}\fz_i
\le x_i<2^{\ell}(\fz_i+1),\,\, i=1,\dots,d\}\,;$$
these sets form a grid of disjoint cubes with sidelength $2^{\ell}$ covering $\bbR^d$. In the following proposition 
we use the conclusion of Proposition \ref{main2} as our hypothesis.

\begin{proposition} \label{technicalprop}
Suppose that, for some $p_1\in (1,2)$,
$$
\Big\|\sum_{j\ge \ell+2} \int_{I_j} \psi*\sigma_r* \eta* F_{j}(r,\cdot) \,
dr\Big\|_{p_1} \lc\Big(\sum_{j\ge 1} 2^{jd}
\|F_{j}\|_{L^{p_1}(\cH)}^{p_1}
\Big)^{1/{p_1}}.
$$
Let
$b_{j,\fz}\in L^2(\cH)$ with $\|b_{j,\fz}\|_{L^2(\cH)}
\le 1$, let
$\beta_j(\fz)\in \bbC$
and define
$$S_j\beta_j(x)=\sum_\fz \beta_j(\fz) \Big(\psi*\eta*
 \int_{I_j} \sigma_r*
\big(\chi_{R^{\ell}_{\fz}}
b_{j,\fz}(r,\cdot) \big)\,dr \Big) \,
$$
Then, for $1<p<p_1$ and $p\le q\le \infty$,
$$
\Big\|\sum_{j\ge  \ell+2} 2^{-jd/p}S_j\beta_j
\Big\|_{L^{p,q}}\\
 \le C_p 2^{\ell d(1/p-1/2) -\eps(p))} \Big(\sum_{\fz\in\mathbb{Z}^d}
\Big(\sum_{j\ge 1}|\beta_{j}(\fz)|^q\Big)^{p/q}\Big)^{1/p}\,,
$$
where  $\eps(p)=\frac{(d-1)p_1'}{2}(\frac 1p-\frac 1{p_1})$.
\end{proposition}

\begin{proof}
We argue as in  \cite{hns2},  Prop. 3.1.
First note that
\Be
\label{p1bound}
\Big\|\sum_{j\ge \ell+2}S_j\beta_j \Big\|_{p_1} \lc
2^{\ell d(1/p_1-1/2)}
\Big(\sum_{j\ge1} 2^{jd} \sum_{\fz}|\beta_j(\fz)|^{p_1}\Big)^{1/p_1}\,.
\Ee
Indeed, by hypothesis the left-hand side is dominated by a constant times
\begin{align*}
&\Big(\sum_{j\ge1} 2^{jd} \Big\|
\sum_\fz \beta_j(\fz)
\chi_{R^{\ell}_{\fz}}
 b_{j,\fz}\Big\|_{L^{p_1}(\cH)}^{p_1}
 \Big)^{1/p_1}
\\&\le \Big(\sum_{j\ge1} 2^{jd} \sum_{\fz}|\beta_j(\fz)|^{p_1}\big\|
\chi_{R^{\ell}_{\fz}}   b_{j,\fz}
\big\|_{L^{p_1}(\cH)}^{p_1}\Big)^{1/p_1}\,
\end{align*}
and after using  H\"older's inequality on each $R^\ell_\fz$ and the $L^2$ normalization of $b_{j,\fz}$ we obtain \eqref{p1bound}.

There is a better $L^1$ bound. Note that for $r\approx 2^j$ the term
$\psi*\sigma_r*b_{j,\fz}(r,\cdot) $ is supported on an annulus with radius
$\approx  2^j$ and width $2^{\ell}$. We use the Cauchy--Schwarz inequality  on this annulus  and then \eqref{Fsigmar} and estimate
\begin{align*}\notag
&\Big\|\sum_{j\ge \ell+2}
\psi*\eta* \sum_\fz
\beta_j(\fz) \int_{I_j} \sigma_r*
(b_{j,\fz}(r,\cdot)\chi_{R^{\ell}_{\fz}})  \,dr
\Big\|_{1} \\\notag &\lc
\sum_{j\ge \ell+2}\sum_\fz  |\beta_j(\fz) | \int_{I_j}
\|\psi*
\sigma_r*
(b_{j,\fz}(r,\cdot)\chi_{R^{\ell}_{\fz}})\|_1  \,
dr
 \\
\notag &\lc
\sum_{j\ge \ell+2}\sum_\fz |\beta_j(\fz) |\int_{I_j} (2^{\ell} 2^{j(d-1)})^{1/2}
\|\psi*
\sigma_r*
(b_{j,\fz}(r,\cdot)\chi_{R^{\ell}_{\fz}}
)\|_2  \,
dr\\&\lc 2^{\ell/2}
\sum_{j\ge 1}2^{j(d-1)}\sum_\fz |\beta_j(\fz) | \int_{I_j} \|b_{j,\fz}(r,\cdot)\|_2 dr \,,
\end{align*}
and  by Cauchy--Schwarz on $I_j$
 and the normalization assumption on $b_{j,\fz}$ we get
\Be\Big\|\sum_{j\ge \ell+2}
S_j\beta_j
\Big\|_{1} \lc
2^{\ell/2}\sum_{j\ge 1} 2^{jd}\sum_\fz |\beta_{j}(\fz)|\,.
\label{L1bd}\Ee
Now   Lemma \ref{weightedinterpol} is used
to interpolate   \eqref{p1bound} and \eqref{L1bd}
and the assertion  follows.
\end{proof}

\section{Proof of Theorem \ref{maxmult}}
We start with a simple fact on  Besov spaces, namely if  $\zeta$ is a
$C^\infty$ function supported on a compact subinterval of $(0,\infty)$ then
\Be\label{besovradial} \|\zeta(|\cdot|) g(|\cdot|)\|_{B^2_{\alpha,q}(\bbR^d)}\lc \|g\|_{B^2_{\alpha,q}(\bbR)}\,,\quad \alpha>0.\Ee
To see this note that the corresponding inequality with Sobolev spaces $L^2_{\alpha}$, $\alpha=0,1,2,\dots$ 
is true by direct computation, and then \eqref{besovradial} follows by real interpolation.

Next if $\cF^{-1}[m(|\cdot|](x)= \kappa(|x|)$ we can use polar coordinates to see that
\Be \label{kappabound}
\big\|m(|\cdot|)\|_{B^2_{\alpha,q}(\bbR^d)} \approx
\Big(
\sum_{j=0}^\infty \Big[ \int_{I_j}|\kappa(r)|^2 r^{2\alpha+d-1} dr\Big]^{q/2}\Big)^{1/q}\,;
\Ee here, as in \S\ref{convsphmeas},  $I_j=[2^j, 2^{j+1}]$ for $j\ge 1$, and  $I_0=(0,2]$.

We shall first prove a dual version of a bound for a maximal operator where the dilations are restricted to $[1,2]$.
\begin{proposition} \label{dualmax}
Let $d\ge 2$, $1<p<\frac{2(d+1)}{d+3}$, $\alpha=d(\frac1p-\frac12)$, $p\le q\le \infty$.
Then, for $m\in B^2_{\alpha,q}$ with support in $(1/2,2)$, 
\Be\label{singlescaleest}
\Big\|\int_{1}^2 T_{m(t\cdot)} f_t\,dt\Big\|_{L^{p,q}}
\lc \|m\|_{B^2_{\alpha,q}} \Big\|\int_1^2|f_t|\,dt\Big\|_{p}.
\Ee
\end{proposition}
\begin{proof}
Let 
$\phi$ be a radial $C^\infty$-function so that $\widehat \phi$ is supported 
 in $\{1/8\le|\xi|\le 8\}$ and equal to one in  $\{1/4\le |\xi|\le 4\}$.
Then 
$T_{m(t\cdot)} f_t=
T_{m(t\cdot)} (\phi*f_t)$ for $1\le t\le 2$.
Also
$$T_{m(t\cdot)} f 
=\int_0^\infty \kappa(r) t^{1-d} \sigma_{rt} *\phi* f \, dr, \,\,\,\text{ $1\le t\le 2$,}
$$
where $\kappa$ is bounded and smooth, and the right-hand side of 
\eqref{kappabound} is finite with $\alpha=d/p-d/2$. We may split $\phi=\psi*\eta$ where $\psi\in C^\infty_c$ 
with $\widehat \psi$ vanishing of high order at the origin. It then  suffices to show that
\begin{multline} \label{splittingterms}
\Big\|\int_1^2\int_{2}^\infty \kappa(r) t^{1-d} \psi*\sigma_{rt} *f_t \, dr\,dt
\Big\|_{L^{p,q}} \\
\lc \Big(\sum_{j=1}^\infty\Big(\int_{I_j}|\ka(r)|^2 r^{2d/p} \frac {dr}{r} \Big)^{q/2}\Big)^{1/q} 
\,\Big\|\int_1^2 |f_t| \, dt \Big\|_p
\end{multline}
This estimate follows by applying  Corollary \ref{mainlor2}.
Take $\nu$ to be  Lebesgue measure on $[1,2]$,  use the tensor product 
$$F_{j,t}(r,x)= 2^{jd/p} \chi_{I_j}(r)\ka(r) \, t^{1-d} f_t(x)$$ and observe that
$\|F_j\|_{L^p(L^1(\cH))}$ can be estimated by the right-hand side of  \eqref{splittingterms}.
\end{proof}

We also need a standard \lq orthogonality\rq \  estimate, in Lorentz spaces.
\begin{lemma}\label{Lpellp}
Let $\{\beta_k\}_{k\in \bbZ}$ a family of $L^1$-functions, satisfying

(i) $\sup_k\|\beta_k\|_{L^1(\bbR^d)}<\infty$,
 
(ii) $\sup_\xi \sum_{k\in \bbZ}|\widehat \beta_k(\xi)| <\infty$. 

\noindent  Then 
\begin{equation}\label{combscales}
\Big\|\sum_k \beta_k*f_k \Big\|_{L^{p,q}}
\lc \Big(\sum_k \|f_k\|_{L^{p,q}} ^p\Big)^{1/p}\,, \quad\text{$1<p<2$, \  $p\le q\le \infty$,}
\end{equation}
and 
\begin{equation}\label{dualcombscales}
\Big(\sum_k \|\beta_k*  f\|_{L^{p,q}}^p\Big)^{1/p}
\lc \|f\|_{L^{p,q}}\,, \quad\text{$2<p<\infty$, \  $1\le q\le p$.}
\end{equation}
Here the functions 
$\{f_k\}$ are allowed to have values in a 
Hilbert space $\sH$ (and $f$ may have  values in $\sH'$).
\end{lemma}

\begin{proof}
By duality  \eqref{combscales} and \eqref{dualcombscales} are equivalent. 
To see \eqref{combscales} 
we define   $\fm_d$ 
to be the 
 product measure on $\Bbb R^{d}\times \bbZ$ 
of Lebesgue measure on $\bbR^{d}$ and
counting measure on $\bbZ$. Define an operator $P$ acting on
functions  $(x,k)\mapsto f_k(x)$, letting $F=\{f_k\}$, by
$PF=\sum_{k}\beta_k*f_k$.
By assumption (i)  $P$  maps the space
$L^1(\Bbb R^{d}\times \bbZ,\fm_d)$ to $L^1(\bbR^{d})$
 and by the almost orthogonality assumption (ii)  it maps
$L^2(\Bbb R^{d}\times \bbZ,\fm_d)$ to $L^2(\bbR^{d})$. Hence by real interpolation
$P$ maps
$L^{p,q}(\Bbb R^{d}\times \bbZ,\fm_d)$ to $L^{p,q}(\bbR^{d})$ for all $1<p<2$ and $q>0$. Let
$$E_{k,m}(F)=\{x:|f_k(x)|_{\sH}>2^m\} .$$
If $p\le q$ we have, by the  triangle inequality in $\ell^{q/p}$,
\begin{multline*}
\|F\|_{L^{p,q}(\fm_d;\sH)}\, \approx \Big(\sum_{m}2^{mq}
\Big|\sum_{k}\meas(E_{k,m}(F))
\Big|^{\frac qp}\Big)^{\frac 1q}\\ \le
\Big(\sum_k\Big(\sum_m2^{mq}
\Big|\meas(E_{k,m}(F))\Big|^{\frac qp}\Big)^{\frac pq}\Big)^{\frac{1}{p}}\,
\approx
(\sum_k \|f_k\|_{L^{p,q}(\sH)} ^p)^{1/p}\,,
\end{multline*}
where for $q=\infty$ we make the usual modification.
This proves \eqref{combscales}.
\end{proof}

\subsubsection*{Proof of Theorem \ref{maxmult}, conclusion}  Now let $\frac{2(d+1)}{d-1}<p<\infty$ and  $p\le q\le\infty$.  
Let $\phi$ be as above and define 
 $L_k$ by $\widehat {L_k f}(\xi)= \widehat \phi(2^{-k}\xi)\widehat f(\xi)$. We may then  estimate 
\[
\|M_m f\|_p \le \Big(\sum_{k\in \bbZ}
\big\|\sup_{1\le t\le 2} |T_{m(2^k t\cdot)} L_k f| \big\|_p^p\Big)^{1/p}\,.
\]
For every $k\in \bbZ$,
$$\big\|\sup_{1\le t\le 2} |T_{m(2^k t\cdot)} L_k f| \big\|_p \le C\|m\|_{B^2_{d/2-d/p, q}} \|L_kf\|_{L^{p,q'}}\,;$$
this follows for  $k=0$ by duality  from  Proposition \ref{dualmax}, and then 
for general $k$ by scaling. 
By Lemma \ref{Lpellp} 
$$
\Big(\sum_{k\in \bbZ}
\big\| L_k f \big\|_{L^{p,q'}}^p\Big)^{1/p} \lc \|f\|_{L^{p,q'}}$$
and combining the estimates we are done. \qed


\section{Proofs of Theorems \ref{sqthm} and \ref{fmthm}}
\label{combscalessect}
Many endpoint bounds for convolution operators
on Lebesgue spaces can be obtained by interpolation
involving a Hardy space estimate and an $L^2$ estimate; this idea goes back to
\cite{steinannbmo}, \cite{fs-hardy}. In some
instances it has been  advantageous to use
Hardy space or $BMO$ methods
such as atomic decompositions or the Fefferman--Stein $\#$-maximal function
directly on $L^p$ to prove theorems which cannot
immediately be obtained by interpolation (see for example endpoint questions
treated in \cite{se-stud}, \cite{stw},  \cite{lrs}, \cite{hns},
\cite{prs}). We
formulate such a result suitable for application in the proofs of Theorems
\ref{sqthm} and \ref{fmthm}. In order to
give a unified treatment we need to consider
vector-valued operators.

Let $\sH_1$, $\sH_2$ be Hilbert spaces. We  consider translation invariant
operators mapping $L^2(\sH_1)$ to $L^2(\sH_2)$, with convolution kernels
having values
in the space $\cL(\sH_1,\sH_2)$ of bounded operators from $\sH_1$ to $\sH_2$.
 On the Fourier transform side,
the operators are given by $\widehat {Tf}(\xi)
= M(\xi) \widehat f(\xi)$ where
$\widehat  f(\xi)\in \sH_1$, $\widehat {Tf}(\xi)\in \sH_2$, with
$\sup_\xi |M(\xi)|_{\cL(\sH_1,\sH_2)}<\infty$.
If $S$ is an $L^2(\bbR^d)$ convolution operator with scalar kernel (and multiplier) and $\sH$ is a Hilbert space then  $S$ extends to a bounded operator on
$L^2(\bbR^d, \sH)$, denoted temporarily by $S\otimes Id_\sH$.
 If $T$ is as before with $\cL(\sH_1,\sH_2)$-valued kernel then
$(S\otimes Id_{\sH_2})T= T(S\otimes Id_{\sH_1})$. With a slight abuse of
notation we shall continue to write $S$ for either
$S\otimes Id_{\sH_2}$ and $S\otimes Id_{\sH_1}$.

We need to  formulate a hypothesis which will be used for
  convolution operators with multipliers  compactly supported away from the origin.

\begin{hypothesis}\label{hypothesis} Let $1<p<2$, $p\le q\le \infty$, $\eps>0$ and $A>0$.
We say that the kernel $\sK$ satisfies
 $\text{Hyp}(p,q,\eps, A)$
if
for every $\ell\ge 0$ one can split the kernel
into a short and long range contribution
$$
\sK= \sK_\ell^{\sh}+\sK_\ell^{\lg}$$ so that  the following properties hold:

(i) $\sK_\ell^{\sh}$ is supported in $\{x:|x|\le 2^{\ell+10}\}$.

(ii) $\sup_{\xi\in \bbR^d} \big|\cF[\sK_\ell^\sh](\xi)|_{\cL(\sH_1,\sH_2)}\le A$.

(iii) For every
family of $L^2$ functions $\{a_\fz\}_{\fz\in \bbZ^d}$,
with $\supp\,( a_\fz)\in R^\ell_\fz$ and  $\sup_\fz \|a_{\fz}\|_{L^2(\sH_1)}\le 1$, and
for $\gamma\in \ell^p(\bbZ^d)$ the inequality
$$\Big\|\sum_{\fz} \sK_\ell^{\lg}* (\gamma(\fz) a_{\fz})\Big\|_{L^{p,q}}
\le A 2^{\ell (d(\frac 1p-\frac 12)-\eps)} \Big(\sum_\fz |\gamma(\fz)|^p\Big)^{1/p}\,$$
holds.
\end{hypothesis}

\begin{theorem} \label{atomic} Given $p\in (1,2)$, $p\le q\le \infty$, $\eps>0$ and $A>0$
suppose that  $\sK^k$, $k\in \bbZ$  are $\cL(\sH_1,\sH_2)$-valued
 kernels satisfying hypothesis
 $\text{Hyp}(p,q,\eps, A)$.
Define  the convolution operator $T_k$ by
$$T_k f(x)= \int 2^{kd} \sK^k(2^k(x-y))f(y) dy  $$
Let $\eta$ be a scalar Schwartz function with $\widehat \eta$ supported in $\{\xi:1/4\le |\xi|\le 4\}$ and let $\eta_k=2^{kd}\eta(2^k\cdot)$.
Then  the operator $f\mapsto \sum_{k\in \bbZ} \eta_k* T_k f$, initially
defined on $\sH_1$ valued Schwartz functions with
 compact Fourier support away from the origin,  extends to  an operator
acting on  all $f\in L^p(\sH_1)$ so that the inequality
$$\Big\|\sum_k \eta_k*T_k f\Big\|_{L^{p,q}(\sH_2)}\le C_p A
\|f\|_{L^p(\sH_1)}$$
holds.
\end{theorem}

The proof of Theorem \ref{atomic} is by now quite standard, but for
completeness  we include it in Appendix \ref{atomicpf}
 below.  Given Theorem \ref{atomic} we now show how it can be used to
deduce  Theorems \ref{sqthm} and \ref{fmthm} from  the results in \S\ref{convsphmeas}.

\begin{remark}\label{p1remark} We actually prove a slightly more general result: Assuming that the estimate of Proposition~\ref{main2} holds for some exponent  $p_1\in (1, \frac{2d}{d+1})$ then the
conclusion of Theorem \ref{fmthm} holds for $1<p<p_1$ and the 
conclusion of  Theorem \ref{sqthm} holds for $p_1'<p<\infty$. A similar remark also applies to Theorem \ref{maxmult}.
\end{remark}

\subsection*{Proof of Theorem \ref{sqthm}}
With $p_1$ as in Remark \ref{p1remark},
 by duality and changes of variables $t= 2^{k} s$ it is enough to show  that, for $1<p<p_1$ and
$\alpha=d(\frac 1p-\frac 12)$,
\Be\label{globalsq}
\Big\|\sum_{k\in \bbZ} \int_1^2 \cF^{-1}\Big[\frac{|\xi|^2}{2^{2k}s^2}
\Big(1-\frac{|\xi|^2}{2^{2k}s^2}\Big)_+^{\alpha-1}\widehat f_s\Big] \frac{ds}{s}
\Big\|_{L^{p,2}}\lc \Big\|\Big(\int_1^2|f_s|^2 \frac{ds}{s}\Big)^{1/2}\Big\|_p\,.
 \Ee
Let $\phi$ be  such that $\widehat \phi$ is supported in $\{1/4\le |\xi|\le 4\}$
with $\widehat \phi(\xi)=1$ in $\{1/3\le|\xi|\le 3\}$.
Let \Be\label{besselalpha}
\Jscr_\alpha(\rho)= \rho^{-\frac{d-2}{2}-\alpha} J_{\frac{d-2}{2}+\alpha}(\rho)
\Ee
so that
$\cF[\Jscr_\alpha(t|\cdot|)](\xi)= c_\alpha t^{-d}
 (1-|\xi|^2/t^2)_+^{\alpha-1}$ (see Chapter VII of \cite{stw}).  In particular $\Jscr_0=\Jscr$ as in
 \eqref{besselfct}.  Let $\phi_k=2^{kd}\phi(2^k\cdot)$.
Then \eqref{globalsq}  follows from
\Be\label{globalmod}\Big\|\sum_{k\in \bbZ} \phi_k*\int_1^2
\int \Jscr_\alpha(s|y|) f_s(\cdot-y) dy \frac{ds}{s}
\Big\|_{L^{p,2}}\lc \Big\|\Big(\int_1^2|f_s|^2 \frac{ds}{s}\Big)^{1/2}\Big\|_p\,.
 \Ee
The reduction of \eqref{globalsq} to \eqref{globalmod} involves incorporating irrelevant powers of $s$ in the definition of $f_s$ and an
 application of standard estimates for vector-valued singular integrals (\cite{stein-si})
to handle the contribution of
$(1-|\xi|^2)^{\alpha-1}_+$ away from the unit sphere. We omit the details.

We now split $\phi=\eta*\psi*\psi$ where $\widehat \eta$ has the same support as $\widehat \phi$ and $\psi$ is a radial $ C^\infty_0$ function
supported in $\{x:|x|\le 1/10\}$, furthermore
$\widehat \psi$ vanishes to order $10 d$ at the origin.
If $\sH_1= L^2([1,2], dr)$ then we wish to apply
 Theorem \ref{atomic} with the $\sH_1'$ valued  kernel
$\sK^k\equiv \sK$ (independent of $k$) defined
by
\Be \label{sKdef}
\inn{\sK(x)}{v} =\int_1^2 v(s)
\int_0^\infty \Jscr_\alpha(sr) \psi*\sigma_r (x)dr \, ds\,.
\Ee
We define the  corresponding
 short range kernel $\sK^{\sh}_\ell$ by
letting the $r$-integral in \eqref{sKdef}
 extend over $[0, 2^{\ell+2}]$ and the long range kernel
$\sK_\ell^{\lg}$ by
letting the $r$-integral extend over $(2^{\ell+2},\infty)$.

Clearly the support condition (i) in Hypothesis   \ref{hypothesis} holds.
Note that $d/p-d/2>1/2$ for $p<2d/(d+1)$.
Thus to check condition (ii) of Hypothesis   \ref{hypothesis}
it  suffices to  verify that
$$\sup_{\xi\in \bbR^d} \Big(\int_1^2\Big|\int_0^{2^{\ell+2}} \Jscr_\alpha(rs)
\widehat\psi(\xi)^2\widehat \sigma_r(\xi) dr\Big|^2
ds\Big)^{1/2} \le A_\alpha\,,\quad\text{  $\alpha>1/2$ }\,.$$
Writing $\widehat\psi(\xi)=u(|\xi|)$, this reduces to
\Be\label{shortrangeassu}
\sup_{\rho>0}
|u(\rho)|^2 \Big(\int_1^2\Big|
\int_0^{2^{\ell+2}} \Jscr_\alpha(rs) \Jscr(r\rho) r^{d-1} dr \Big|^2 ds\Big)^{1/2}
\lc A_\alpha\,.
\Ee
We may take the $r$-integral over $[1,2^{\ell+2}]$ since the estimate for
 the contribution for $r\in [0,1]$ is  immediate. We use
the standard asymptotic expansions for the modified Bessel-function
 $\Jscr_\alpha$,
\Be\label{besselasy}\Jscr_\alpha(u)=
u^{-\frac {d-1}2-\alpha}
\Big[\sum_{n=0}^1 u^{-n}
\big(c_{n,\alpha}^+ e^{iu} +
c_{n, \alpha}^{-}  e^{-iu}\big)
 +
O(|u|^{-2})\Big], \quad u\ge 1\,
\Ee
and also the analogous expansion for  $\Jscr=\Jscr_0$.
If we consider   only the leading terms in both asymptotic expansions
we are led to bound
$$\sup_{\rho>0}
\frac{|u(\rho)|^2}{\rho^{\frac{d-1}{2}}} \Big(\int_1^2\Big|
\int_1^{2^{\ell+2}}   e^{ir(\pm s\pm \rho)} r^{-\alpha} dr \Big|^2 ds\Big)^{1/2}
\lc A_\alpha, \quad \alpha>1/2\,,
$$
which follows from Plancherel's theorem on $\bbR$.
The other terms with lower order or  nonoscillatory error terms are similar or  more straightforward.
Note that we also use $|u(\rho)|\le \rho^{10 d}$ for $\rho\in (0,1)$. This establishes condition (ii) in Hypothesis \ref{hypothesis}.

Finally we verify condition (iii).
Let $\{a_{\fz}\}_{\fz\in \bbZ^d}$ be $L^2(\sH_1)$ functions with $\sup_{\fz}\|a_{\fz}\|_{L^2(\sH_1)}\le 1$,  supported on
$2^{\ell}$-cubes with disjoint interiors.  We then need to show that
\begin{multline} \label{longrangeassu}\Big\|\sum_{j\ge\ell+2} \int_{I_j}  \psi*\psi*\sigma_r *
\sum_{\fz}
\gamma(\fz) \int_1^2 \Jscr_\alpha(sr)
 a_{\fz}(s,\cdot) ds
\Big\|_{L^{p,2}}\\
\lc A 2^{\ell (d(\frac 1p-\frac 12)-\eps)} \Big(\sum_\fz |\gamma(\fz)|^p\Big)^{1/p}\,.
\end{multline}
Setting $$c_{j,\fz}=\Big(\int_{I_j}\int_{R_\fz}\Big|
\int_1^2 \Jscr_\alpha(sr)
 a_{\fz}(s,x) ds
\Big|^2 \frac{dr}{r} dx\Big)^{1/2}$$
 we may apply
Proposition \ref{technicalprop}
for $q=2$ with
$$\beta_j(\fz)= 2^{jd/p}\gamma(\fz) c_{j,\fz}\,, \quad\text{ and  }\quad
b_{j,\fz}(r, x) =\chi_{I_j}(r) c_{j,\fz}^{-1}
\int_1^2 \Jscr_\alpha(sr)
 a_{\fz}(s,x) ds$$ if
$c_{j,\fz}\neq 0$ and
$b_{j,\fz}=0 $ if $c_{j,\fz}=0$.
We can then  dominate  the left-hand side of \eqref{longrangeassu}  by
a constant times
$$2^{\ell(\frac dp-\frac d2-\eps(p))}
\Big(\sum_{\fz}\Big(\sum_j|\beta_j(\fz)|^2\Big)^{p/2}\Big)^{1/p}
\,$$
with $\eps(p)>0$ for $p<p_1$.
We are only left to show that for fixed $\fz$
$$\Big(\sum_j|\beta_j(\fz)|^2\Big)^{1/2} \lc |\gamma(\fz)|\,$$
where the implicit constant is uniform in $\fz$.
This estimate follows from
\Be\label{betaestfizedz}
\sum_{j\ge\ell+2} 2^{2jd/p}\int_{I_j}\Big|\int_1^2 \Jscr_\alpha(sr)
 a_{\fz}(s,x) ds\Big|^2 \frac{dr}{r} \lc\int_1^2|a_{\fz}(s,x)|^2 ds
\Ee
and  integration over $x\in R_\fz$. To see \eqref{betaestfizedz} we
use again the asymptotics \eqref{besselasy}.
The estimate for the oscillatory terms (with $n=0,1$) becomes
\begin{multline*}\sum_{j\ge\ell+2} 2^{2jd/p}\int_{I_j}
r^{-2\alpha-2n-d}\Big|\int_1^2 e^{\pm isr} s^{-\frac{d-1}{2}-\alpha-n}  a_{\fz}(s,x) ds\Big|^2 dr
\\ \, \lc\, \int_1^2|a_{\fz}(s,x)|^2 ds\,,
\end{multline*}
and  since $\alpha=d/p-d/2$ it suffices to show
$$\int\Big|\int_1^2 e^{\pm is r} v(s)  a_{\fz}(s,x) dt\Big|^2 dr \lc\int_1^2|a_{\fz}(s,x)|^2 ds$$
with $\sup_s|v(s)|\le C$. But this  is an immediate consequence of
Plancherel's theorem. Lastly, if  in \eqref{betaestfizedz}
we put
the error term $O((sr)^{-\alpha-\frac{d-1}{2}-2})$ for   $\Jscr_\alpha(sr)$
the resulting expression can be easily estimated by
$$ \int_{2^\ell}^\infty r^{-3} dr \Big[ \int_1^2|
a_{\fz}(s,x)| ds\Big]^2\lc\int_1^2|a_{\fz}(s,x)|^2 ds\,.$$
This concludes the proof of \eqref{betaestfizedz}, and thus the proof of Theorem \ref{sqthm}.
\qed

\subsection*{Proof of Theorem \ref{fmthm}}
We apply Theorem \ref{atomic} with $\sH_1=\sH_2=\bbC$.
It is easy to see that it suffices
to show that, for $\alpha=d(1/p-1/2)$,
$$\Big\|\sum_{k\in \bbZ} \cF^{-1}\big[ m_k(2^{-k}|\cdot|) \widehat \eta(2^{-k}\cdot)\widehat f\big]\Big\|_{L^{p,q}}\ls
\sup_k\|m_k\|_{B^2_{\alpha,q}} \|f\|_p,
$$
where $m_k$ are functions  in $B^{2}_{\alpha,q}(\bbR)$
supported in $(1/2,2)$ and $\eta$ is a radial Schwartz function with $\widehat \eta$ supported in the annulus
$\{1/4<|\xi|<4\}$.
Now  write $\cF^{-1}[m_k(|\cdot|) ](x) = \kappa_k(|x|).$
Using polar coordinates and \eqref{besovradial} we see that
\Be\label{kernelestbesov}
\Big(
\sum_{j\ge 1} \Big[ \int_{I_j}|\kappa_k(r)|^2 r^{2d/p} \frac{dr}{r}\Big]^{q/2}\Big)^{1/q}
\lc \|m_k\|_{B^2_{\alpha,q}}, \quad \alpha= d(1/p-1/2)\,,
\Ee
and of course $\sup_{0<r\le 1}|\kappa_k(r)|<\infty$.
With $\psi$ as  in the proof of Theorem \ref{sqthm} it suffices  to show that
 the kernels
$$
K^k_\ell=  K^{k,\sh}_\ell+K^{k,\lg}_\ell
=  \Big[\int_0^{2^{\ell+2}}+\int_{2^{\ell+2}}^\infty\Big]
\kappa_k(r)\, \psi*\psi*\sigma_r \, dr
$$ satisfy the assumptions of Hypothesis \ref{hypothesis}, uniformly in $k$.
Note that by \eqref{Fsigmar}
\begin{align*}
\Big|\cF\Big[ \int_{I_j}\kappa_k(r)\, \psi*\psi*\sigma_r  \,
dr\Big](\xi)\Big|\,\lc\,\int_{I_j}|\kappa_k(r)| r^{\frac{d-1}{2}} dr &
\\
\lc 2^{-j(\frac dp-\frac {d+1}2)}
 \Big(\int_{I_j}|\kappa_k(r)|^2 r^{2d/p} \frac{dr}{r}\Big)^{1/2}&
\end{align*} and since $p<\frac{2d}{d+1}$ we may sum in $j$ to deduce
that $\sup_k\|\widehat K^{k,\sh}_\ell\|_\infty<\infty$.

We turn to the kernels $K_\ell^{k,\lg}$  and again show
using Proposition \ref{technicalprop}
that they suffice condition (iii) in Hypothesis
\ref{hypothesis}.
Define $$\beta_{k,j}(\fz)= \gamma(\fz)\Big(\int_{I_j}|\ka_k(r)|^2 r^{2d/p}\frac {dr}r\Big)^{1/2}$$ and $$b_{k,j,\fz}(r,x) = 2^{-1} [\beta_{k,j}(\fz)]^{-1} 2^{jd/p} \chi_{I_j}(r) \kappa_k(r) a_{\fz}(x)$$
if $\beta_{k,j}(\fz)\neq 0$  (and $b_{k,j,\fz}=0$ otherwise).
Then $$\|b_{k,j,\fz}\|_{L^2(\cH)}= \Big(\int\int_0^\infty|b_{j,z}(r,x)|^2 \frac{dr}{r} dx\Big)^{1/2}\le 1.$$
Now $$K^{k,\lg}_\ell*\sum_\fz \gamma(\fz)a_\fz
= \sum_\fz \sum_{j\ge \ell+2}
\beta_{k,j}(\fz)
\int_{I_j} \psi*\psi*\sigma_r *
b_{k,j,\fz}(r,\cdot) dr$$ and  by   Proposition \ref{technicalprop} we have
\begin{align*}
\Big\|K^{k,\lg}_\ell*\sum_\fz a_\fz\Big\|_p \lc 2^{\ell(\frac dp-\frac d2-\eps(p))}
\Big(\sum_\fz\Big(\sum_j|\beta_{k,j}(\fz)|^q\Big)^{p/q}\Big)^{1/p},
\end{align*}
with $\eps(p)>0$ for $p<p_1$.
Finally, by \eqref{kernelestbesov}
$$\Big(\sum_\fz\Big(\sum_j|\beta_{k,j}(\fz)|^q\Big)^{p/q}\Big)^{1/p}
\lc \Big(\sum_\fz|\gamma(\fz)|^p\Big)^{1/p} \|m_k\|_{B^{2}_{d(1/p-1/2),q}},
$$
which completes the proof.\qed

\appendix
\section{Proof of Theorem \ref{atomic}}\label{atomicpf}\
By normalization we may assume that Hypothesis $\text{Hyp}(p,q,\eps,A)$ holds with $A=1$. We use atomic decompositions in $L^p$  which are
 constructed from  square functions, based on the ideas
by Chang and Fefferman \cite{crf}.
A convenient and useful form is given by an $\ell^2$-valued version of
Peetre's   maximal square function
({\it cf.} \cite{peetre}, \cite{triebel}),
$$
\fS f (x) = \Big(\sum_k \sup_{|y|\le 100 d\cdot 2^{-k}}
|\cL_k f(x+y)|_{\cH_1}^2\Big)^{1/2},
$$
where $\cL_k f=\phi_k*f$, with $\phi_k=2^{kd}\phi(2^k\cdot)$, and $\phi$ is a radial
Schwartz function  with $\widehat \phi$ supported in $\{\xi:1/5<|\xi|<5\}$.
Then $$\|\fS f\|_p\le C_p \|f\|_{L^p(\sH_1)}\,, \quad  1<p<\infty\,.$$

We closely follow the argument in \cite{hns2}. Choose $\phi$ by splitting
the function~$\eta$  in the statement of Theorem \ref{atomic} as
$$\eta=\psi*\phi$$
where $\psi$ is a radial $C^\infty_0$-function with support in $\{x:|x|<1/4\}$
whose Fourier transform vanishes to order $10 d$ at the origin.
We set $\psi_k=2^{kd}\psi(2^k\cdot)$, then $\eta_k=\psi_k*\phi_k$ and we have
$$\sum_k  \eta_k T_k f= \sum_k  \psi_k T_k \cL_k f\,.$$

For  $k\in \bbZ$, we tile $\bbR^{d}$ by the dyadic cubes of sidelength
$2^{-k}$ and   write $L(Q)=-k$
if  the sidelength of a dyadic cube $Q$ is $2^{-k}$.
For each $n\in \bbZ$, let
$$\Omega_n=\{x: \fS f(x)>2^n\}.$$ Let $\cQ^n_{-k}$ be the set
of all dyadic cubes  of sidelength $2^{-k}$ which have the property that
$|Q\cap \Omega_n|\ge |Q|/2$ but
$|Q\cap \Omega_{n+1}|<|Q|/2$.
Let
$$ \Omega_n^*= \{x: M \chi_{\Omega_n}(x)>100^{-d}\}$$
with  $M$  the Hardy--Littlewood maximal operator. The set $\Omega_n^*$ is  open,
contains  $\Omega_n$ and satisfies  $|\Omega_n^*|\lc |\Omega_n|$.
Let $\cW_n$ be  the set of all dyadic cubes $W$ for which the $50$-fold dilate of $W$ is contained in $\Omega_n^*$
and $W$ is maximal with respect to this property.
The collection $\{W\}$ forms a Whitney-type decomposition of $\Omega_n^*$.
The  interiors of the Whitney cubes  are disjoint.
For each
 $W\in\cW_n$ we denote by $W^*$ the tenfold
dilate of $W$; the   dilates $\{W^*: W\in
\cW_n\}$ have still  bounded overlap.

Note that each $Q\in \cQ^n_{-k}$ is contained in a unique $W\in \cW_n$.
For
$W\in \cW_n$,  set
$$a_{k,W,n}
= \sum_{\substack{Q\in\cQ^n_{-k}\\Q\subset W}} (\cL_kf) \chi\ci Q\,,$$
and for any dyadic cube $W$ define
$$a_{k,W}=\sum_{n: W\in \cW_n} a_{k,W,n}.$$
The functions  $a_{k,W,n} $ can be considered as \lq atoms\rq, but
without the usual normalization.
For fixed $n$ one has
 \Be \label{L2atomicest}
\sum_{W\in \cW_n} \sum_{k}\|a_{k,W,n}\|_{L^2(\sH_1)}^2 \lc 2^{2n} \meas(\Omega_n).
\Ee
Indeed (arguing as in \cite{crf}) the left-hand side is equal to
\begin{align*}&\sum_{Q\in \cW_n}
\sum_k \sum_{Q\in \cQ^n_{-k}}\int_Q |\cL_k f(x)|_{\sH_1}^2 \, dx
\\
&\le 2 \sum_{Q\in \cW_n}\sum_k \sum_{Q\in \cQ^n_{-k}}\int_{Q\cap(\Omega_n\setminus
 \Omega_{n+1})} \sup_{|y|\le 2^{-k}\sqrt d} |\cL_k f(x+y)|_{\sH_1}^2 \, dx
\\&\le
 2\int_{\Omega_n \setminus\Omega_{n+1}} \fS f(x)^2  \le 2 \, \meas(\Omega_n) 2^{2(n+1)}\,.
\end{align*}

Let 
$T_{k,\ell}^\lg$, $T_{k,\ell}^\sh$,
  be the convolution operator with kernels $2^{kd}\sK^{k,\lg}_\ell(2^k\cdot)$
and  $2^{kd}\sK^{k,\sh}_\ell(2^k\cdot)$, respectively.
The desired estimate will follow once we establish the short range  inequality
\begin{equation}\label{shortr}
\Big\|\sum_k \sum_{\ell\ge 0}\sum_{\substack{W\in\cup_n \cW_n\\
L(W)=-k+\ell}}
\psi_k*T_{k,\ell}^\sh  a_{k,W}\Big\|_{L^{p}(\sH_2) }  \lc  \|\fS f\|_p\,.
\end{equation}
and for fixed  $\ell\ge 0$ the long range inequality
\begin{equation}\label{longr}
\Big\|\sum_k \sum_{\substack{W\in \cup_n\cW_n\\
L(W)=-k+\ell}}
\psi_k*T_{k,\ell}^\lg  a_{k,W}
\Big\|_{L^{p,q}(\sH_2) }  \lc 2^{-\ell\eps} \|\fS f\|_p\,.
\end{equation}

\subsubsection*{Proof of \eqref{shortr}}
We  prove that for $1<r<2$ and for fixed $n\in \bbZ$
\Be\label{shortrfixedj}
\Big\|\sum_k
 \sum_{\ell\ge 0}\sum_{\substack{W\in \cW_n\\
L(W)=-k+\ell}}
\psi_k*T_{k,\ell}^\sh  a_{k,W,n}
\Big\|_{L^{r}(\sH_2)}^r  \le C_r \,
\,2^{nr} \meas (\Omega_n).
\Ee
By \lq real interpolation\rq \  (cf. Lemma  2.2 in \cite{hns}) it follows that the stronger estimate
\Be\notag
\Big\|
\sum_{n} \sum_k\sum_{\ell\ge 0}\sum_{\substack{W\in \cW_n\\
L(W)=-k+\ell}}
\psi_k*T_{k,\ell}^\sh  a_{k,W,n}
\Big\|_{L^{p}(\sH_2)}^p  \lc \,
\sum_n\,2^{np} \meas (\Omega_n).
\Ee
holds and this implies \eqref{shortr} since  $\sum_n 2^{np}
\meas(\Om_n)\lc\|\fS f\|_p^p$.

Since the expression inside the norm in \eqref{shortrfixedj} is supported  in
$\Omega_n^*$ we see that the  left-hand side of \eqref{shortrfixedj} is dominated by
\Be
\label{shortrfixedjL2}
\meas (\Omega_n^*)^{1-r/2}\Big\|\sum_k
 \sum_{\ell\ge 0}\sum_{\substack{W\in \cW_n\\
L(W)=-k+\ell}} \psi_k*T_{k,\ell}^\sh  a_{k,W,n}
\Big\|_{L^{2}(\sH_2)}^r\,.
\Ee
The convolution operators with kernel $\psi_k$  are
 almost orthogonal and thus we can dominate the left-hand side of
 \eqref{shortrfixedjL2} by a constant times
\begin{multline}
\label{shortrfixedjL2orth}
\meas (\Omega_n^*)^{1-r/2}\Big(\sum_k\Big\|
 \sum_{\ell\ge 0}\sum_{\substack{W\in \cW_n\\
L(W)=-k+\ell}}
T_{k,\ell}^\sh  a_{k,W,n}
\Big\|_{L^{2}(\sH_2)}^2\Big)^{r/2}\,.
\end{multline}
Now,  for each  $W$ with $L(W)=-k+\ell$,
the function $T_{k,\ell}^\sh  a_{k,W,n}$ is
supported in the expanded cube $W^*$.
The
cubes $W^*$ with $W\in \Omega_j$  have bounded overlap, and therefore
the expression
\eqref{shortrfixedjL2orth} is
\Be
\label{shortrsumW}
\lc\meas (\Omega_n^*)^{1-r/2}
\Big(\sum_k\sum_{\ell\ge 0}\sum_{\substack{W\in \cW_n\\
L(W)=-k+\ell}}
\Big\|
T_{k,\ell}^\sh  a_{k,W,n}
\Big\|_{L^{2}(\sH_2)}^2
\Big)^{r/2}\,.
\Ee
Now we have for fixed $W$
$$\big\|T_{k,\ell}^\sh  a_{k,W,n}
\big\|_{L^{2}(\sH_2)}\lc \|  a_{k,W,n}\|_{L^2(\sH_1)}.$$
By  \eqref{L2atomicest} we have
\[\sum_k\sum_{\ell\ge 0}\sum_{\substack{W\in \cW_n\\
L(W)=-k+\ell}}
\big\| a_{k,W,n}\big\|_2^2 \lc
\sum_{W\in \cW_n}\sum_k\| a_{k,W,n}\big\|_2^2 \lc
2^{2n}\meas(\Omega_n).\]
Since $\meas(\Omega_n^*)\lc
\meas(\Omega_n)$ it follows that the right-hand side of
\eqref{shortrsumW}
is dominated by a constant times
$ \meas(\Omega_n) 2^{nr}$
which then yields \eqref{shortrfixedj} and finishes the proof of the
short range estimate.

\subsubsection*{Proof of \eqref{longr}}
We use the first estimate in 
Lemma \ref{Lpellp}, with  $\beta_k=\psi_k$, and the $\sH_2$ valued functions
  $F_k= \sum_{\substack{W:\\L(W)=-k+\ell}}
 T^\lg_{k,\ell}
 a_{k,W}$.
We then see that \eqref{longr} follows from
\begin{equation}
\label{longrsumjelllp}
\sum_k\Big\|\sum_{\substack{W:\\L(W)=-k+\ell}}
 T^\lg_{k,\ell}
 a_{k,W}
\Big\|_{L^{p,q}(\sH_2) }^p  \lc  2^{- \ell \eps p }
 \sum_n\meas(\Omega_n) 2^{np}\,.
\Ee
By rescaling and assumption (iii) in Definition \ref{hypothesis}
we have
for every $k$
\begin{multline*}
\Big\|T^\lg_{k,\ell}\big[ \sum_{\substack{W:\\
L(W)=-k+\ell}}
  a_{k,W}\big]\Big\|_{L^{p,q}(\sH_2) }  \\
\lc  2^{-\ell  \eps}
2^{(\ell-k) d(\frac 1p-\frac 12)}
\Big(\sum_{\substack{W:\\L(W)=-k+\ell}}
\|a_{k,W}\|_{L^2(\sH_1)}^p \Big)^{1/p}\,.
\end{multline*}
Thus in order to finish the proof we need the inequality
\Be\label{Lpatomineq}
\sum_{k}
\sum_{\substack{W:\\L(W)=-k+\ell}} 2^{(\ell-k) d(1-\frac p2)}
\|a_{k,W}\|_{L^2(\sH_1)}^p  \lc \sum_n 2^{np} \meas(\Omega_n)\,.
\Ee
For fixed $k$ and fixed $W$,
 the functions $ a_{k,W,n}$, $n\in \bbZ$ live on disjoint sets (since the dyadic
 cubes of sidelength  $2^{-k}$ are disjoint and each such cube
 is in exactly one
 family $\cQ_{-k}^n$).
Therefore
$$\|a_{k,W}\|_{L^2(\sH_1)} \lc\Big(\sum_n\|a_{k,W,n}\|^2_{L^2(\sH_1)}\Big)^{1/2}$$
and thus we can
bound the left-hand side
of \eqref{Lpatomineq} by
\begin{align*}
&\sum_n \sum_{k\in \bbZ}\sum_{\substack{W\in \cW_n:\\
L(W)=-k+\ell}}
2^{(\ell-k) d(1-\frac p2)}
\|a_{k,W,n}\|_{L^2(\sH_1)}^p
\\
&\le \sum_n \Big(\sum_{k\in \bbZ}
\sum_{\substack{W\in \cW_n:\\L(W)=-k+\ell}}   \meas(W)\Big)^{1-p/2}
\Big(\sum_k\sum_{\substack{W\in \cW_n:\\
L(W)=-k+\ell}}
\|a_{k,W,n}\|_{L^2(\sH_1)}^2 \Big)^{p/2}
\\
&\le \sum_n  \meas(\Omega_n^*)^{1-p/2}
\Big(\sum_k\sum_{W\in \cW_n}
\|a_{k,W,n}\|_{L^2(\sH_1)}^2 \Big)^{p/2}\,;
\end{align*}
here we used the  disjointness of Whitney cubes in $\cW_n$.
By  \eqref{L2atomicest} the last displayed expression
is bounded by
$$
C \sum_n  \meas(\Omega_n^*)^{1-p/2} (2^{2n}\meas(\Omega_n))^{p/2} \lc \sum_n 2^{np}\meas(\Omega_n)
$$
which gives \eqref{Lpatomineq}.
\qed

\end{document}